\documentclass[letterpaper,10pt,reqno,onefignum,onetabnum]{amsart}
\usepackage[english]{babel}
\usepackage{amsmath}
\usepackage{amsthm}
\usepackage{verbatim}
\usepackage{mathrsfs}
\usepackage{bm}
\usepackage[foot]{amsaddr}
\usepackage[dvipsnames]{xcolor}
\usepackage{hyperref}
\hypersetup{
	colorlinks = true,
	linkcolor = OliveGreen,
	anchorcolor = OliveGreen,
	citecolor = OliveGreen,
	filecolor = OliveGreen,
	urlcolor = OliveGreen
}
\usepackage{algorithm}% http://ctan.org/pkg/algorithms
\usepackage{algorithmic}% http://ctan.org/pkg/algorithms
\usepackage{float}
\usepackage{lipsum}
\usepackage{amsfonts}
\usepackage{amssymb}
\usepackage{graphicx}
\usepackage{epstopdf}
\usepackage{cases}
\usepackage{multirow}
\ifpdf
\DeclareGraphicsExtensions{.eps,.pdf,.png,.jpg}
\else
\DeclareGraphicsExtensions{.eps}
\fi
\usepackage{verbatim}
\usepackage{mathrsfs}
\usepackage{bm}

\newtheorem{teo}{Theorem}[section]
\newtheorem{prop}[teo]{Proposition}

\newtheorem{cor}[teo]{Corollary}
\newtheorem{pro}[teo]{Problem}
\newtheorem{algo}[teo]{Algorithm}

\newtheorem{rem}[teo]{Remark}
\newtheorem{ejem}[teo]{Example}

\usepackage{lipsum}
\usepackage{amsfonts}
\usepackage{amssymb}
\usepackage{graphicx}
\usepackage{epstopdf}
\usepackage{cases}
\usepackage{multirow}
\usepackage{algorithmic}
\usepackage{tikz}
\usepackage{booktabs}%
\usetikzlibrary{matrix}
\usetikzlibrary{arrows}
\ifpdf
\DeclareGraphicsExtensions{.eps,.pdf,.png,.jpg}
\else
\DeclareGraphicsExtensions{.eps}
\fi
\usepackage{verbatim}
\usepackage{mathrsfs}
\usepackage{bm}
\usepackage{color}
\usepackage{caption}
\usepackage{subcaption}
\usepackage{enumitem}
\newcommand{\N}{\mathbb N}

\newcommand{\R}{\mathbb R}

\renewcommand{\H}{\mathcal{H}}
\newcommand{\G}{\mathcal G}

\newcommand{\sH}{\mathsf H}
\newcommand{\sG}{\mathsf G}
\newcommand{\sx}{\mathsf x}

\newcommand{\sg}{\mathsf g}
\newcommand{\sw}{\mathsf w}
\newcommand{\sy}{\mathsf y}

\renewcommand{\sp}{\mathsf p}

\newcommand{\su}{\mathsf u}

\newcommand{\sT}{\mathsf T}

\newcommand{\sL}{\mathsf L}
\newcommand{\sV}{\mathsf V}

\newcommand{\HH}{{\bm{\mathcal{H}}}}

\newcommand{\id}{\textnormal{Id}}
\newcommand{\x}{\bm x}
\newcommand{\y}{\bm y}

\newcommand{\weak}{\rightharpoonup}
\newcommand{\ran}{\textnormal{ran}\,}
\newcommand{\dom}{\textnormal{dom}\,}
\newcommand{\infconv}{\ensuremath{\mbox{\small$\,\square\,$}}}

\newcommand{\zer}{\textnormal{zer}}

\newcommand{\gra}{\textnormal{gra}\,}

\newcommand{\argm}[1]{\underset{#1}{\argmin\, }}

\newcommand{\scal}[2]{{\left\langle{{#1}\mid{#2}}\right\rangle}}

\newcommand{\menge}[2]{\big\{{#1}~\big |~{#2}\big\}}
\newcommand{\pinf}{\ensuremath{{+\infty}}}

\newcommand{\RPP}{\ensuremath{\left]0,+\infty\right[}}

\newcommand{\RX}{\ensuremath{\left]-\infty,+\infty\right]}}

\newcommand{\sri}{\ensuremath{\text{\rm sri}\,}}

\newcommand{\prox}{\ensuremath{\text{\rm prox}\,}}

\newcommand{\weakly}{\ensuremath{\:\rightharpoonup\:}}
\usepackage{geometry}
\numberwithin{equation}{section}

\DeclareFontEncoding{FMS}{}{}
\DeclareFontSubstitution{FMS}{futm}{m}{n}
\DeclareFontEncoding{FMX}{}{}
\DeclareFontSubstitution{FMX}{futm}{m}{n}
\DeclareSymbolFont{fouriersymbols}{FMS}{futm}{m}{n}
\DeclareSymbolFont{fourierlargesymbols}{FMX}{futm}{m}{n}
\DeclareMathDelimiter{\nr}{\mathord}{fouriersymbols}{152}{fourierlargesymbols}{147}

\DeclareMathOperator*{\argmin}{arg\,min}
\DeclareMathDelimiter{\nr}{\mathord}{fouriersymbols}{152}{fourierlargesymbols}{147}
\DeclareMathAlphabet{\mathpzc}{OT1}{pzc}{m}{it}

\title[FRB and SDR with 
partial inverses]{Forward-Reflected-Backward and 
	Shadow-Douglas--Rachford with partial inverse for Solving 
	Monotone 
	Inclusions}

\author{Fernando Rold\'an$^{\dagger}$}
\address{$^{\dagger}$ Departamento de Ingeniería Matemática, Universidad de Concepción, Concepción, Chile. 
	{\it 
		E-mail address:} 
	{\sf{fernandoroldan@udec.cl, 
			fdoroldan94@gmail.com}}. }
\begin{document}
	
	\begin{abstract}
		In this article, we study two methods for
			solving
			monotone inclusions in real Hilbert spaces involving the 
			sum of a 
			maximally monotone operator, a monotone-Lipschitzian 
			operator, a cocoercive 
			operator, and a normal cone to a vector subspace. Our 
			algorithms split and exploits the intrinsic properties of 
			each
			operator involved in the inclusion. We derive our methods by 
			combining partial inverse techniques 
			with the {\it forward-reflected-backward} algorithm  
			and with the {\it shadow-Douglas--Rachford} algorithm, 
			respectively. Our 
			methods inherit 
			the advantages of those methods, requiring only one 
			activation 
			of the Lipschitzian operator, one activation of the 
			cocoercive operator, 
			two projections onto the closed vector subspace, and one 
			calculation of the resolvent of the maximally monotone 
			operator. Additionally, to allow larger step-sizes in one of the proposed methods, we revisit FSDR by extending its convergence for larger step-sizes. Furthermore, we provide methods for solving monotone inclusions involving a sum of maximally monotone operators and for 
				solving a system of primal-dual inclusions
			involving a mixture	of sums, linear compositions, 
			parallel 
			sums, Lipschitzian operators, cocoercive operators, and 
			normal cones. We apply our methods to constrained composite 
			convex optimization problems as a specific example. Finally, 
			in order to compare our methods with existing methods 
			in the literature, 
			we provide numerical experiments on constrained total 
			variation least-squares optimization problems and computed tomography inverse problems. We obtain 
			promising numerical results.
		\par
		\bigskip
		\noindent \textbf{Keywords.} {\it Splitting algorithms, 
			monotone operator theory, partial inverse, convex 
			optimization, forward-reflected-backward, 
			shadow-Douglas--Rachford}
		\par
		\bigskip \noindent
		2020 {\it Mathematics Subject Classification.} {47H05, 47J25, 
			49M29, 65K05, 90C25.}
	\end{abstract}
	
	\maketitle
	
\section{Introduction}
This article is devoted to the numerical resolution of the 
following 
problem.
\begin{pro}\label{prob:main}
	Let $A: \H \to 2^\H$ be a maximally monotone operator, let
	$B : \H \to \H$ be a $\beta$-Lipschitzian operator for some 
	$\beta >0$, let $C: \H \to \H$ be a $\zeta^{-1}$-cocoercive 
	operator 
	for some 
	$\zeta >0$ and let $V \subset \H$ be a closed vector 
	subspace. 
	The 
	problem is to 
	\begin{equation}\label{eq:problemmain}
		\text{find} \quad x \in \H \quad \text{such that} \quad 0 
		\in 
		Ax+Bx+Cx+N_Vx,
	\end{equation}{}
	where $N_V$ denotes the normal cone to $V$ and under the 
	assumption that its solution set, denoted by  $Z$, is 
	not 
	empty.
\end{pro}
This problem is deep related with mechanical problems 
\cite{Gabay1983,Glowinski1975,Goldstein1964}, differential 
inclusions \cite{AubinHelene2009,Showalter1997}, game theory 
\cite{AttouchCabot19,Nash13}, image processing  
\cite{chambolle1997,Pustelnik2019,daubechies2004}, traffic theory 
\cite{Nets1,Bui2022traffic,Fukushima1996The,GafniBert84}, among others 
engineering 
problems. 

Problem~\ref{prob:main} can be solved by the {\it forward-partial 
	inverse-half-forward} splitting (FPIHF) proposed in 
\cite{BricenoRoldanYuchaoChen}. Additionally, in particular 
instances such as $V=\H$, or when some of the operator vanish, 
methods mentioned in \cite{BricenoRoldanYuchaoChen} can also 
solve Problem~\ref{prob:main}. FPIHF requires, at each iteration, 
two activation of $B$, 
one activation of $C$, three 
projection 
onto the subspace $V$, and one backward step by calculating the 
resolvent of $A$. This algorithm is derived by 
combining partial 
inverses 
techniques \cite{Spingran1983AMO,Spingarn1985MP} and the 
{\it forward-backward-half-forward} (FBHF) splitting 
proposed in \cite{BricenoDavis2018}. In turn, FBHF can solve 
Problem~\ref{prob:main} when $V=\H$ and generalize the 
{\it forward-backward} (FB) 
\cite{Goldstein1964,Lions1979SIAM,passty1979JMAA} and the 
{\it forward-backward-forward} or {\it Tseng's 
	splitting} (FBF) \cite{Tseng2000SIAM}. When $V=\H$, FB and 
FBF can 
solve 
Problem~\ref{prob:main} when $B=0$ and $C=0$, respectively. FBHF
iterates as follows
\begin{equation}
	\label{eq:algoFBHF}
	(\forall n \in \N)\quad \left\lfloor
	\begin{aligned}
		&z_{n}=J_{{\gamma}A} (x_n-\gamma (B+C)x_n)\\
		&x_{n+1}=x_n-\gamma (Bz_n-Bx_n),
	\end{aligned}
	\right.
\end{equation}
where $x_0 \in \H$ and its weak convergence to a zero of $A+B+C$ 
is 
guaranteed for $\gamma 
\in ]0,4/(\zeta+\sqrt{\zeta^2+16\beta^2})[$. In the case when 
$B=0$ the recurrence on \eqref{eq:algoFBHF} reduces to 
FB, whose convergence is guaranteed for  $\gamma \in 
]0,2/\zeta[$.  In the case when 
$C=0$ the recurrence on \eqref{eq:algoFBHF} reduces to 
FBF, whose convergence is guaranteed for  $\gamma \in 
]0,1/\beta[$.
From \eqref{eq:algoFBHF} can notice that FBHF require one 
activation of $C$ and two 
activations of $B$, hence, FPIHF inherits this 
property. Methods generalizing FB and FBF which handle the case 
where $V\neq\H$ are proposed in \cite{briceno2015Optim} and 
\cite{Briceno2015JOTA}, respectively. Similarly, by using partial 
inverses 
techniques and the Primal-Dual algorithm \cite{Condat13,Vu13}, 
the authors in \cite{BricenoDeride2023} proposed a method which 
can solve Problem~\ref{prob:main} involving an additional 
backward step
calculating the resolvent of the Lipschitzian operator.

Recently, several methods arise for 
finding a 
zero of $A+B$ which, by 
storing the 
previous step, require only one activation of $B$ per iteration
\cite{Cevher2020SVVA,Csetnek-2019-AMO,Malitsky2020SIAMJO}. In 
particular, the authors in \cite{Malitsky2020SIAMJO} propose the 
{\it forward-reflected-backward} splitting (FRB) and an
extension of it, called 
{\it forward-half-reflected-backward} (FHRB), which can 
solve Problem~\ref{prob:main} when $V=\H$. FHRB iterates
\begin{equation}
	\label{eq:MTCOR}
	(\forall n \in \N)\quad \left\lfloor
	\begin{aligned}
		&x_{n+1}=J_{{{\lambda}}A} (x_n-{\lambda}
		(2Bx_n-Bx_{n-1}+Cx_{n})),\\
	\end{aligned}
	\right.
\end{equation}
where, $x_0 \in \H$ and its weak convergence to a zero of 
$A+B+C$ is guaranteed for ${\lambda}
\in ]0,2/(4\beta+\zeta)[$. The recurrence in \eqref{eq:MTCOR} reduces 
to FRB when, $C=0$ and it reduces to FB when $B=0$. 
Inertial and momentum variants of FHRB are proposed in 
\cite{MorinBanertGiselsson2022,Tang2022FHRBin}.
On the other hand, the authors in 
\cite{Csetnek-2019-AMO}, by a discretization of a dynamical 
system associated to the Douglas--Rachford operator 
\cite{Eckstein1992,Lions1979SIAM,Svaiter2011}, derive the {\it 
	shadow-Douglas--Rachford}algorithm (SDR) for finding a zero 
of $A+B$. Later, the authors in \cite{Jingjing20213OPSDR}, include 
inertia to SDR  and also extend it finding a zero of 
$A+B+C$, we call this extension {\it 
	backward-shadow-Douglas--Rachford} (FSDR). FSDR iterates as 
\begin{equation}
	\label{eq:SDRCOR}
	(\forall n \in \N)\quad \left\lfloor
	\begin{aligned}
		&x_{n+1}=J_{{\lambda}A} (x_n-{\lambda}(B+C)x_n)-{\lambda}
		(Bx_n-Bx_{n-1}),
	\end{aligned}
	\right.
\end{equation}
where, $x_0 \in \H$ and its weak convergence to a zero of $A+B+C$ 
is guaranteed for ${\lambda}\in ]0,2/(3\zeta+9\beta)[$. In the case 
when 
$C=0$, the recurrence in \eqref{eq:SDRCOR} reduces to SDR. Note 
that,
if $\zeta \to 0$, then $2/(3\zeta+9\beta)\to 2/(9\beta)$. 
However, 
the authors in \cite{Csetnek-2019-AMO} guarantee the convergence 
of 
SDR for $\lambda\in 
]0,1/(3\beta)[$, therefore, the result in 
\cite{Jingjing20213OPSDR} 
no longer recover the step-sizes of SDR.
In this paper,  by proceeding similar to \cite{BricenoRoldanYuchaoChen}, that is, combining partial inverse 
techniques with FHRB and FSDR, we propose splitting 
algorithms 
that 
fully exploits the structure and properties of each operator 
involved in Problem~\ref{prob:main}. Our methods require only 
one 
activation
of the Lipschitzian operator, one activation of the 
cocoercive operator, two projections onto the closed vector
subspace, and one computation of the resolvent of the maximally 
monotone operator. In the particular case when 
$V=\H$ we recover FHRB and FSDR. Additionally, we provide a 
result of convergence allowing larger step-sizes in FSDR, being a result of interest in itself.
By applying our algorithms in a product space, we derive two
algorithms solving a system of primal-dual inclusions involving a mixture 
of sums, 
linear 
compositions, parallel sum, Lipschitzian operators, cocoercive 
operators, 
and 
normal cones. Additionally, we propose two methods for solving a monotone inclusion that involves the sum of maximally monotone operators avoiding the incorporation of dual variables, which can be problematic in large dimension problems. These methods are derived by applying the proposed algorithms in an adequate subspace $V$.
As a particular instance, 
we tackle examples in composite 
convex optimization problems under vector subspace constraints.
In order to compare the proposed methods with each other and with 
existing methods in the literature, we test our algorithms in a 
constraint linear composite TV-regularized least-squares problem and in image reconstruction problems in computed tomography. These numerical experiments exhibit the numerical advantages of the proposed methods over the method proposed in \cite{BricenoRoldanYuchaoChen}, FHRB, and Condat--V\~u \cite{Condat13,Vu13} thus highlighting the practical importance of our proposed methods.
The article is organized as follows. In 
Section~\ref{sec:notation} we 
present our notation and classical results in monotone 
operator 
theory and convex analysis. In Section~\ref{sec:mr} we derive the 
converge of our algorithms for solving Problem~\ref{prob:main}. 
In 
Section~\ref{sec:app} we apply our methods for solving 
composite 
primal-dual inclusions, convex optimization problems, and a inclusion involving a sum of maximally monotone operators. Finally, 
Section~\ref{sec:numerical} is 
devoted to numerical experiments in total 
variation least-squares optimization problems and compute tomography image reconstruction.

\section{Notations and Preliminaries} \label{sec:notation}
Throughout this paper, $\H$ and $\G$ are real Hilbert spaces. We 
denote their scalar 
products by $\scal{\cdot}{\cdot}$, the associated norms by 
$\|\cdot \|$, and the weak and strong convergence by $\weakly$ 
and $\to$, respectively.
Given a linear bounded 
operator $L:\H \to \G$, we denote its adjoint by 
$L^*\colon\G\to\H$. 
$\id$ denotes the identity operator on $\H$. 
Let $D\subset \H$ be non-empty, let $T: D \rightarrow \H$, and 
let $\beta \in \left]0,+\infty\right[$. The operator $T$ is 
$\beta-$cocoercive if 
\begin{equation} \label{def:coco}
	(\forall x \in D) (\forall y \in D)\quad \langle x-y \mid 
	Tx-Ty 
	\rangle 
	\geq \beta \|Tx - Ty 
	\|^2
\end{equation}
and it is $\beta-$Lipschitzian if 
\begin{equation} \label{def:lips}
	(\forall x \in D) (\forall y \in D)\quad \|Tx-Ty\| \leq  
	\beta\|x 
	- y 
	\|.
\end{equation}	
Let $A:\H \rightarrow 2^{\H}$ be a set-valued operator. The 
domain, 
range, and graph of $A$ are 
$\dom\, A = \menge{x \in \H}{Ax \neq  \varnothing}$, 
$\ran\, A = \menge{u \in \H}{(\exists x \in \H)\, u \in Ax}$, 
and $\gra 
A = \menge{(x,u) \in \H \times \H}{u \in Ax}$,
respectively. 
The set of zeros of $A$ is  $\zer A = 
\menge{x \in \H}{0 \in Ax}$, the inverse of $A$ is $A^{-1}\colon 
\H  
\to 2^\H \colon u 
\mapsto 
\menge{x \in \H}{u \in Ax}$, and the resolvent of $A$ is 
given by $J_A=(\id+A)^{-1}$. Let $B:\H\to 2^\H$ be a set valued 
operator, 
the 
parallel sum 
of $A$ and $B$ is given by $A\infconv B:=(A^{-1}+B^{-1})^{-1}$.
The operator $A$  is monotone if 
\begin{equation}\label{def:monotone}
	(\forall (x,u) \in \gra A) (\forall (y,v) \in \gra A)\quad 
	\scal{x-y}{u-v} 
	\geq 0.
\end{equation}
Additionally, $A$ is maximally monotone if it is monotone and 
there exists 
no 
monotone operator $B :\H\to  2^{\H}$ such that $\gra B$ properly 
contains $\gra A$. Given a non-empty closed convex set 
$V\subset\H$, we denote by 
$P_V$ the projection onto $V$. In the case when $V$ 
is a 
closed vector subspace of $\H$ we have that $P_V$ is a linear 
bounded 
operator and that $P_{V^\bot} = \id- P_V$, where ${V^\bot}$ 
denotes
the orthogonal complement of $V$. The partial inverse of 
$A$ 
with respect to a closed vector subspace $V \subset \H$, denoted 
by 
$A_V$, is 
defined by
\begin{equation}\label{eq:defparinv}
	(\forall (x,y) \in \H^2) \quad y \in A_V x \quad 
	\Leftrightarrow\quad  
	(P_V y + P_{V^{\bot}}x) \in A(P_V x +P_{V^{\bot}}y).
\end{equation} 
We denote by $\Gamma_0(\H)$ the class of proper lower 
semicontinuous convex functions $f\colon\H\to\RX$. Let 
$f\in\Gamma_0(\H)$.
The Fenchel conjugate of $f$ is 
defined by 
\begin{equation*}
	f^*\colon u\mapsto \sup_{x\in\H}(\scal{x}{u}-f(x)),
\end{equation*} 
which 
is a function in $\Gamma_0(\H)$. The subdifferential of $f$ is 
the maximally monotone operator
$$\partial f\colon x\mapsto \menge{u\in\H}{(\forall y\in\H)\:\: 
	f(x)+\scal{y-x}{u}\le f(y)}.$$
We have that $(\partial f)^{-1}=\partial f^*$ 
and that $\zer\,\partial f$ is the set of 
minimizers of $f$, which is denoted by $\arg\min_{x\in \H}f$. 
The proximity operator of $f$ is given by
\begin{equation}
	\label{e:prox}
	\prox_{f}\colon 
	x\mapsto\argm{y\in\H}\Big(f(y)+\frac{1}{2}\|x-y\|^2\Big).
\end{equation}
We have 
$\prox_f=J_{\partial f}$. Moreover, it follows from \cite[Theorem 
14.3]{bauschkebook2017} that
\begin{equation}
	\label{e:Moreau_nonsme}
	\prox_{\gamma f}+\gamma \prox_{f^*/\gamma} \circ 
	\id/\gamma=\id.
\end{equation}
Given a non-empty closed convex set $C\subset\H$, we denote by 
$\iota_C\in\Gamma_0(\H)$ the indicator function of $C$, which 
takes the value $0$ in $C$ and $\pinf$ otherwise, and by $N_C= 
\partial (\iota_C)$ the normal cone to $C$. 

For further properties of monotone operators,
non-expansive mappings, and convex analysis, the 
reader is referred to \cite{bauschkebook2017}.

\section{Main Results}\label{sec:mr}
This section is dedicated to the numerical resolution of 
Problem~\ref{prob:main}.
We propose the Algorithm~\ref{algo:MTC} and the
Algorithm~\ref{algo:SDR}, which are derived by combining 
partial inverses 
techniques with FHRB and with FSDR, respectively. We divide this 
section 
into two subsections, each containing one algorithm. In the first 
subsection, we present 
Algorithm~\ref{algo:MTC} and its result of convergence. In the 
second 
subsection, we provide a convergence result for FSDR, allowing 
for 
larger step-sizes. Next, we derive the convergence of 
Algorithm~\ref{algo:SDR}, similar to that of 
Algorithm~\ref{algo:MTC}.

\subsection{Forward-Reflected-Partial inverse-Backward Splitting 
} 
On this subsection, we propose the following algorithm for 
solving 
Problem~\ref{prob:main}.
\begin{algo}\label{algo:MTC}
	In the context of 
	Problem~\ref{prob:main}, 
	let $(x_{-1},x_{0},y_{0})\in V\times V\times V^{\bot}$, let 
	$\gamma \in\RPP$, define $w_{-1} = B 
	x_{-1}$, and consider 
	the recurrence
	\begin{equation}
		\label{eq:algoMTC}
		(\forall n \in \N)\quad \left\lfloor
		\begin{aligned}
			&w_n = B x_n\\
			&u_n = x_{n}+\gamma 
			y_{n}-\gamma P_V (2 w_n - w_{n-1}+Cx_n)\\
			&p_{n}=J_{{\gamma}A} u_n\\
			&x_{n+1}=P_Vp_n\\
			&y_{n+1}=y_{n}-\frac{p_n-x_{n+1}}{\gamma}.
		\end{aligned}
		\right.
	\end{equation}
\end{algo}
Next, we present the result of convergence for 
Algorithm~\ref{algo:MTC}. The proof follows a similar idea of combining partial inverses with FBHF in the proof of Theorem~3.2 in \cite{BricenoRoldanYuchaoChen}. Here, we combine partial inverses with FHRB.
\begin{teo} \label{teo:MTC}
	In the context of Problem~\ref{prob:main}, let $\gamma\in 
	\RPP$, 	let $(x_{-1},x_{0},y_{0})\in V\times 
	V\times 
	V^{\bot}$, 
	let 
	$w_{-1} = P_VBx_{-1}$,
	and let 
	$(x_n)_{n 
		\in 
		\N}$ and 
	$(y_n)_{n \in \N}$ be the sequences generated 
	Algorithm~\ref{algo:MTC}.
	Suppose that,
	\begin{equation}\label{eq:gamma}
		\gamma \in 	
		\left]0,\frac{2}{4\beta+\zeta}\right[.			
	\end{equation}
	Then,  there exist
	$\overline{x}\in \zer(A+B+C+N_V)$ and $\overline{y}\in 
	V^{\bot}\cap(A\overline{x}+P_{V}(B+C)\overline{x})$
	such that $x_{n}\weak \overline{x}$ and $y_{n}\weak 
	\overline{y}$.
\end{teo}

\begin{proof}
	First, define the operators:
	\begin{equation}\label{eq:defOp}
		\begin{cases}
			&\mathcal{A}_\gamma = (\gamma A)_V \colon \H \to 
			2^\H,\\
			&\mathcal{B}_\gamma = \gamma P_V \circ B \circ P_V 
			\colon 
			\H 
			\to \H,\\
			&\mathcal{C}_\gamma = \gamma P_V \circ C \circ P_V 
			\colon 
			\H 
			\to \H.
		\end{cases}
	\end{equation}
	Note that, by 
	\cite[Proposition 3.1]{Briceno2015JOTA}, 
	$\mathcal{A}_\gamma$ is maximally monotone 
	and $\mathcal{B}_\gamma$ is $\gamma\beta$-Lipschitzian. 
	Additionally, 
	by \cite[Proposition~5.1(ii)]{briceno2015Optim}, 
	$\mathcal{C}_\gamma$ is $(\gamma \zeta)^{-1}$-cocoercive. 
	Now, fix 
	$n\in\N$ and define $q_n = (u_n-p_n)/\gamma$. Thus, we 
	have $u_n=p_n+\gamma q_n$. Moreover, it follows from 
	\eqref{eq:algoMTC} that $q_n \in \gamma A 
	p_n$ and that
	\begin{align*}
		P_Vp_n+\gamma P_{V^\bot}q_n &= P_V p_n  - 
		P_{V^\bot} p_n  +  
		P_{V^\bot} u_n\\
		& = 2P_V p_n  - p_n  + (\id-P_V)u_n\\
		& = 2P_V J_{\gamma A}u_n  - J_{\gamma A}u_n  + 
		(\id-P_V)u_n\\
		& = (2P_V\circ J_{\gamma A} - 
		J_{\gamma A}+\id - P_V)u_n.
	\end{align*}
	Then, it follows from \cite[Proposition 
	3.1(i)]{Briceno2015JOTA} when 
	$\delta=1$, that 
	\begin{equation}\label{eq:rsvJA}
		P_Vp_n+\gamma P_{V^\bot}q_n=	
		J_{\mathcal{A}_\gamma} 
		u_n.
	\end{equation}
	Now, since $x_0 \in V$ and $y_0 \in V^{\bot}$, it follows 
	from 
	\eqref{eq:algoMTC}  that $(x_n)_{n \in \N}$ and 
	$(y_n)_{n \in \N}$ are sequences in $V$ and $V^{\bot}$, 
	respectively. Then, if we define $z_n = x_n +\gamma y_n$, 
	since $P_{V^{\bot}} u_n = \gamma 
	y_n$, we 
	have
	that $P_V z_{n+1} = x_{n+1} = P_Vp_n$ and that 
	\begin{align*}
		P_{V^{\bot}} 
		z_{n+1} &=  P_{V^{\bot}} (\gamma y_{n+1})\\
		&= P_{V^{\bot}} (\gamma y_{n}-(p_n-x_{n+1}))\\
		&= \gamma y_n - P_{V^{\bot}} p_n\\
		&= \gamma y_n - P_{V^{\bot}} (u_n-\gamma q_n)\\
		&= \gamma 
		P_{V^{\bot}}q_n.
	\end{align*}
	Therefore,
	\eqref{eq:algoMTC} and \eqref{eq:rsvJA} yield
	\begin{align*}
		(\forall n \in \N) \quad	z_{n+1} & = P_V z_{n+1} + 
		P_{V^{\bot}} 
		z_{n+1} \\
		& = P_V p_n + \gamma P_{V^{\bot}} q_n \\
		&= J_{\mathcal{A}_\gamma} u_n\\
		&= J_{\mathcal{A}_\gamma} 
		(x_{n}+\gamma y_{n}-\gamma P_V ( 
		2w_n-w_{n-1}+C x_n))\\
		&= J_{\mathcal{A}_\gamma} 
		(x_{n}+\gamma y_{n}-\gamma P_V ( 
		2Bx_n-Bx_{n-1}+C x_n))\\
		&= J_{ \mathcal{A}_\gamma} 
		(z_{n}- (2\mathcal{B}_\gamma 
		z_n-\mathcal{B}_\gamma z_{n-1}+\mathcal{C}_\gamma z_{n})).
	\end{align*} 
	Hence, $(z_n)_{n \in \N}$ is a particular instance of the 
	recurrence 
	in \eqref{eq:MTCOR} when 
	$\lambda=1$. Moreover, it follows from 
	\eqref{eq:gamma} that
	\begin{equation*}
		\lambda =1 \in 
		\left]0,\frac{2}{4\beta\gamma+\zeta\gamma}\right[.
	\end{equation*}  
	Therefore, \cite[Theorem~5.2]{Malitsky2020SIAMJO} implies	
	that $z_n 
	\weak \overline{z}\in 
	\zer({\mathcal{A}_\gamma+\mathcal{B}_\gamma 
		+\mathcal{C}_\gamma 
	})$. Furthermore, since $C$ is $\zeta$-Lipschitzian, 
	$\hat{B}:=B+C$ is $ 
	(\beta+\zeta)$-Lipschitzian, thus, by applying
	\cite[Proposition~3.1(iii)]{Briceno2015JOTA} to
	$\hat{B}$ and 
	$\hat{\mathcal{B}}_\gamma:=\mathcal{B}_\gamma+\mathcal{C}_\gamma$
	we have that $\hat{x} \in \H$ is a 
	solution to Problem~\ref{prob:main} if and only if 
	\begin{multline} \label{eq:solABC}
		\hat{x} \in V \ \text{ and }\ 
		\big(\exists \hat{y} \in V^\bot\cap 
		(A\hat{x}+\hat{B}\hat{x})\big)
		\ 	\text{ such that }\  \hat{x} + 	
		\gamma\big(\hat{y}-P_{V^\bot} 
		\hat{B} 
		\hat{x}\big) \in 
		\zer 
		(\mathcal{A}_\gamma+\hat{\mathcal{B}}_\gamma).
	\end{multline}
	By setting $\overline{x} = P_V 
	\overline{z}$ and $\overline{y} = P_{V^\bot} 
	\overline{z}/\gamma$, 
	we have $-(\mathcal{B}_\gamma+\mathcal{C}_\gamma) 
	(\overline{x}+\gamma\overline{y})\in 
	\mathcal{A}_\gamma(\overline{x}+\gamma\overline{y})$,
	which is equivalent to
	$-P_V (B+C) \overline{x}+\overline{y}\in A\overline{x}$. 
	Hence, 
	by 
	defining 
	$\hat{y}=\overline{y}+P_{V^{\bot}}(B+C)\overline{x}\in 
	V^{\bot}\cap (A\overline{x}+B\overline{x}+C\overline{x})$, we 
	have
	$\overline{x}+\gamma(\hat{y}-P_{V^{\bot}}(B+C)\overline{x})\in
	\zer (\mathcal{A}_\gamma + 
	\mathcal{B}_\gamma+\mathcal{C}_\gamma)$ 
	and \eqref{eq:solABC} implies that 
	$\overline{x} \in Z$ and that $\overline{y}\in 
	V^{\bot}\cap(A\overline{x}+P_{V}(B+C)\overline{x})$.
	Finally, by the weak continuity of $P_V$ and 
	$P_{V^{\bot}}$, 
	we conclude that  $x_n = 
	P_V z_n \weak P_{V} \overline{z} = \overline{x}$ and 
	$y_{n}=P_{V^{\bot}}z_{n}/\gamma\rightharpoonup 
	P_{V^{\bot}}\overline{z}/\gamma=\overline{y}$, which
	completes the proof.
\end{proof}
\begin{rem}\label{rem:1}
	\begin{enumerate}
		\item \label{rem:11} Note that, in comparison with 		
		Algorithm~\ref{algo:MTC}, the 
		method 
		proposed in 
		\cite{BricenoRoldanYuchaoChen} for solving 
		Problem~\ref{prob:main}
		require one additional activation of $B$ and one 
		additional 
		projection onto $V$ per iteration.
		\item\label{rem:12} The proof of Theorem~\ref{teo:MTC} 
		shows that 
		Algorithm~\ref{algo:MTC} is a particular instance of the 
		recurrence in \eqref{eq:MTCOR} when the 
		step-size $\lambda$ is equal to $1$. Hence, by 
		introducing  
		$\lambda  \in 
		]0,2/(\gamma(4\beta+\zeta))[$, for $\gamma \in \RPP$, and 
		considering initial points
		$(x_{-1},x_{0},y_{0})\in V\times V\times V^{\bot}$ and 
		$w_{-1} = B 
		x_{-1}$, Algorithm~\ref{algo:MTC} can be extended to the 
		following routine
		\begin{equation}
			\label{eq:algoMT}
			(\forall\, n \in \N) \quad	\left\lfloor
			\begin{aligned}
				%&\text{for }n=1,2,\ldots\\
				&w_{n}=B x_n,\\
				&\text{find }\, (p_{n},q_{n})\in \mathcal{H}^2\, 
				\text{ 
					such that:}\\
				&p_{n}+\gamma 
				q_{n}=x_{n}+\gamma y_{n}-\lambda\gamma P_V ( 
				2w_n-w_{n-1}-C x_n)  \text{ and } 
				\\\
				& \quad \quad  \quad  \quad 
				\frac{P_{V}q_{n}}{\lambda}+P_{V^{\bot}}q_{n}\in 
				A\left(P_{V}p_{n}+\frac{P_{V^{\bot}}p_{n}}{\lambda}\right),\\
				&x_{n+1}=P_{V}p_{n},\\
				&y_{n+1}=P_{V^{\bot}}q_{n}.
			\end{aligned}
			\right.
		\end{equation}  
		The proof of convergence for this recurrence follows 
		similarly by 
		invoking \cite[Proposition 3.1(i)]{Briceno2015JOTA} for 
		$\delta =\lambda$. The parameter $\lambda$ can 
		be manipulated in order to accelerate the convergence. 
		However,	the 
		inclusion in \eqref{eq:algoMT} is not always easy to 
		solve, 
		hence, we prefer to present Algorithm~\ref{algo:MTC} 
		allowing 
		explicit 
		calculations in terms of the resolvent of $A$.
		\item Note that, in the case when $V=\H$, 
		Algorithm~\ref{algo:MTC} reduces to the method 
		in \eqref{eq:MTCOR} proposed in 
		\cite{Malitsky2020SIAMJO}.
		\item In the case when $C=0$, Algorithm~\ref{algo:MTC} 
		provide a recurrence for finding a zero of $A+B+N_V$  
		inheriting the benefits mentioned in \ref{rem:11} in 
		comparison with the method proposed in 
		\cite{Briceno2015JOTA}. In this case we can take $\zeta 
		\to 0$ 
		and 
		then $\gamma \in ]0,1/(2\beta)[$.
		\item In the case when $B=0$, Algorithm~\ref{algo:MTC} 
		reduced to the method in 
		\cite[Corollary~5.3]{briceno2015Optim} when $\lambda_n 
		\equiv 
		1$. In this we can take $\beta \to 0$ and then $\gamma 
		\in 
		]0,2/\zeta[$.
		\item The authors in \cite{MorinBanertGiselsson2022} 
		derived a momentum version of FHRB. Therefore, by 
		following a similar approach to the proof of 
		Theorem~\ref{teo:MTC}, we can 
		deduce the convergence of the following momentum version 
		of our method
		\begin{equation}
			(\forall n \in \N)\quad \left\lfloor
			\begin{aligned}
				&v_n = (1+\theta)(x_{n}+\gamma
				y_{n})-\theta(x_{n-1}+\gamma 
				y_{n-1})\\
				&w_n = B x_n\\
				&u_n = v_n-\gamma P_V (2 w_n - w_{n-1}+Cx_n)\\
				&p_{n}=J_{{\gamma}A} u_n\\
				&x_{n+1}=P_Vp_n\\
				&y_{n+1}=y_{n}-\frac{p_n-x_{n+1}}{\gamma},
			\end{aligned}
			\right.
		\end{equation}
		where $(x_{-1},x_{0},y_{0})\in V\times 
		V\times 
		V^{\bot}$, $w_{-1} = P_VBx_{-1}$, $\theta \in ]-1,1/3[$, 
		and the 
		step-size satisfy
		\begin{equation*}
			\gamma \in 
			\left]0,\frac{2(1-\theta-2|\theta|)}{4\beta+\zeta}\right[.
		\end{equation*}
		The main drawback of this extension, is that it is forced 
		to smaller 
		the step-sizes, which usually leads to worse numerical 
		results.
	\end{enumerate}
\end{rem}

\subsection{Forward-Partial Inverse-Shadow-Douglas--Rachford}
In this subsection, we study the weak convergence of the 
following 
algorithm for solving Problem~\ref{prob:main}.
\begin{algo}\label{algo:SDR}
	In the context of 
	Problem~\ref{prob:main}, 
	let $(x_{-1},x_{0},y_{0})\in V\times V\times V^{\bot}$, let 
	$\gamma \in\RPP$, define $w_{-1} = B 
	x_{-1}$, and consider 
	the recurrence
	\begin{equation}
		\label{eq:algoSDR}
		(\forall n \in \N)\quad \left\lfloor
		\begin{aligned}
			&w_n = B x_n\\
			&p_{n}=J_{{\gamma}A}\big(x_{n}+\gamma 
			y_{n}-\gamma P_V (w_n+Cx_n)\big)\\
			&x_{n+1}=P_V(p_n - \gamma(w_n - w_{n-1}))\\
			&y_{n+1}=y_{n}-\frac{p_n-x_{n+1}}{\gamma}.
		\end{aligned}
		\right.
	\end{equation}
\end{algo}
First, we present the following result, which allows for larger 
step-sizes in 
FSDR than those proposed in 
\cite[Theorem~3.1]{Jingjing20213OPSDR}. 
For additional details, see Remark~\ref{rem:stepsizes}. This 
result 
is of 
interest in its own right.
\begin{prop}\label{prop:Jingjing}
	In the context of Problem~\ref{prob:main}, let $\lambda \in 
	\RPP$, 
	let $(x_0,x_{-1}) \in \H^2$, and consider the sequence 
	$(x_n)_{n 
		\in \N}$ given by 
	\begin{equation}
		\label{eq:SDRCOR2}
		(\forall n \in \N)\quad \left\lfloor
		\begin{aligned}
			&x_{n+1}=J_{{\lambda}A} (x_n-\lambda 
			(B+C)x_n)-\lambda 
			(Bx_n-Bx_{n-1}).
		\end{aligned}
		\right.
	\end{equation}
	Suppose that $\lambda$ satisfies
	\begin{equation}\label{eq:gammaBSDR}
		\frac{2}{3}-(2\beta+\zeta)\lambda-\beta^2\zeta\lambda^3 > 
		0.
	\end{equation}
	Then, the sequence $(x_n)_{n\in \N}$ converges weakly to a 
	point in 
	$\zer(A+B+C)$.
\end{prop}
\begin{proof}
	Let $x \in \zer (A+B+C)$, set $y =\lambda B x$, and set, for 
	every $n 
	\in 
	\N$, $y_n=\lambda B x_n$ and $z_n=x_n+y_{n-1}$. By proceeding 
	similar to the proof of \cite[Lemma~3]{Csetnek-2019-AMO}, by 
	\eqref{eq:SDRCOR2} 
	and the monotonicity of $B$, we conclude that
	\begin{align}\label{eq:CMTp1}
		\|z_{n+1}-z\|^2 & \leq \|z_n-z\|^2 
		+4\scal{y_n-y_{n-1}}{x_n-x_{n+1}}-\|x_{n+1}-x_{n}\|^2-3\|y_n-y_{n-1}\|^2\nonumber\\
		& \hspace*{2cm}+
		2\lambda\scal{Cx_n-Cx}{x-x_{n+1}}+2\lambda\scal{Cx-Cx_n}{y_n-y_{n-1}}.
	\end{align}
	Now, by the cocoercivity of $C$ and the Young's inequality, 
	we have
	\begin{align}\label{eq:CMTp2}
		2\lambda\scal{Cx_n-Cx}{x-x_{n+1}} & = 
		2\lambda\scal{Cx_n-Cx}{x-x_{n}} + 
		2\lambda\scal{Cx_n-Cx}{x_n-x_{n+1}}\nonumber\\
		&\leq  -\frac{2\lambda}{\zeta}\|Cx_n-Cx\|^2
		+\frac{\lambda}{\zeta}\|Cx_n-Cx\|^2 + \lambda 
		\zeta\|x_{n+1}-x_n\|^2 \nonumber\\
		&\leq  -\frac{\lambda}{\zeta}\|Cx_n-Cx\|^2 + \lambda 
		\zeta\|x_{n+1}-x_n\|^2.
	\end{align}
	Additionally, by the Young's inequality and the Lipschitzian 
	property of $B$, 
	we conclude
	\begin{align}\label{eq:CMTp3}
		2\lambda\scal{Cx-Cx_n}{y_n-y_{n-1}} &\leq 
		\frac{\lambda}{\zeta}\|Cx_n-Cx\|^2+\zeta\lambda\|y_n-y_{n-1}\|^2\nonumber\\
		&\leq 
		\frac{\lambda}{\zeta}\|Cx_n-Cx\|^2+\zeta\beta^2\lambda^3\|x_n-x_{n-1}\|^2
	\end{align}
	%Hence, 
	%\begin{align}\label{eq:CMTp4}
	%	\|z_{n+1}-z\|^2 +\left( \frac{2}{3} - \lambda\zeta - 
	%\lambda\beta 
	%	\right)\|x_{n+1} -x_{n}\|^2& \leq \|z_n-z\|^2 + 
	%	(\lambda\beta+\zeta\beta^2\lambda^3)\|x_{n} -x_{n-1}\|^2.
	%\end{align}
	Hence, by defining $\varepsilon = \frac{2}{3} - \lambda\zeta 
	- 
	2\lambda\beta-\zeta\beta^2\lambda^3>0$, it follows from 
	\eqref{eq:CMTp1}-\eqref{eq:CMTp3} and 
	\cite[Equation~(28) \& Equation~(29)]{Csetnek-2019-AMO} that
	\begin{align}\label{eq:CMTp5}
		\|z_{n+1}-z\|^2 
		+(\lambda\beta+\zeta\beta^2\lambda^3+\varepsilon)\|x_{n+1}-x_{n}\|^2&
		\leq \|z_n-z\|^2 + 
		(\lambda\beta+\zeta\beta^2\lambda^3)\|x_{n} -x_{n-1}\|^2.
	\end{align}
	Additionally, it follows from \eqref{eq:SDRCOR2} that
	\begin{align}\label{eq:CMTp6}
		&-\begin{pmatrix}
			z_{n+1}-z_n-\lambda\big(Cx_n - 
			C(z_{n+1}-z_n+x_n)\big)\\
			z_{n+1}-z_n
		\end{pmatrix} \nonumber\\
		&\hspace*{5cm}\in \left( \begin{bmatrix} \lambda(A+C)\\ 
			(\lambda 
			B)^{-1}\end{bmatrix} + \begin{bmatrix}
			0 & \id \\
			-\id & 0
		\end{bmatrix}\right) \begin{pmatrix}
			z_{n+1}-z_n+x_n\\
			z_{n+1}-x_{n+1}
		\end{pmatrix}.
	\end{align}
	Moreover, by the cocoercivity of $C$, we have
	\begin{equation}\label{eq:CMTp7}
		\|Cx_n - C(z_{n+1}-z_n+x_n)\| \leq \zeta \|z_{n+1}-z_n\|,
	\end{equation}
	then, if $\|z_{n+1}-z_n\| \to 0$, we have $\|Cx_n - 
	C(z_{n+1}-z_n+x_n)\|  \to 0$. Finally, by noticing that $A+C$ 
	is 
	maximally monotone \cite[Corollary 25.5]{bauschkebook2017} 
	and 
	combining \eqref{eq:CMTp5}-\eqref{eq:CMTp7}, the result 
	follows 
	similarly to the proof of 
	\cite[Theorem~3]{Csetnek-2019-AMO}.
\end{proof}
\begin{rem}\label{rem:stepsizes} The step-sizes in
	\cite[Theorem~3.1]{Jingjing20213OPSDR} that guaranteed the 
	convergence of the recurrence in \eqref{eq:SDRCOR2}
	are $\lambda \in ]0,2/(9\beta+3\zeta)[$. We can note that:
	\begin{enumerate}
		\item If  $\lambda \in 
		]0,2/(9\beta+3\zeta)[$, then $\lambda^2 \zeta \beta <1$ 
		and 
		\eqref{eq:gammaBSDR} is 
		directly 
		satisfied.  
		\item In the case when $\zeta = 0$, the condition 
		$\lambda 
		\in 
		]0,2/(9\beta+3\zeta)[$ no longer reduces to the condition 
		$\lambda \in ]0,1/(3\beta)[$ in 
		\cite[Theorem~3]{Csetnek-2019-AMO}.  On the other hand, 
		if 
		$\zeta =0$, from \eqref{eq:gammaBSDR}, we can recover the 
		step-sizes in \cite[Theorem~3]{Csetnek-2019-AMO}.
		\item Since $C$ is cocoercive, it is $\zeta$-Lipschitz 
		and 
		$B+D$ is $(\beta+\zeta)$-Lipschitz. Hence, SDR can be 
		applied 
		for finding a zero of $A+B+C$. 
		In that case, the step-size should satisfy $\lambda < 
		1/(3\zeta+3\beta)$. Note that, if $\beta> \zeta$ we have 
		$2/(9\beta+3\zeta) <1/(3\zeta+3\beta)$, then, SDR allows 
		bigger step-sizes than 
		\cite[Theorem~3.1]{Jingjing20213OPSDR}. On the other 
		hand, in 
		the case when 
		$\beta>\zeta$, we have
		\begin{equation}\label{eq:remgam}
			\frac{2}{3}-(2\beta+\zeta)\lambda-\beta^2\zeta\lambda^3
			>  	\frac{2}{3}-3\beta\lambda-\beta^3\lambda^3.
		\end{equation}
		Note that, by introducing the variable, $x:=\beta\lambda$ 
		the 
		right side of \eqref{eq:remgam} reduces to the 
		function $f \colon x \mapsto 2/3-3x-x^3$. From the 	
		graph\footnote{\url{https://www.geogebra.org/calculator/dhvvswa6}}
		of $f$
		we observe that if $x \in ]0,2/5]$ we have $f(x) > 0$, 
		then 
		\eqref{eq:remgam} is strictly positive. 
		Thus, if $\lambda\beta < 2/5$, from \eqref{eq:remgam} we 
		conclude 
		that $\gamma$ satisfies
		\eqref{eq:gammaBSDR}. 
		Additionally, we have 
		$1/(3\beta+3\zeta)<1/(3\beta)<2/(5\beta)$, 
		therefore, 
		step-sizes 
		satisfying 
		\eqref{eq:gammaBSDR} are larger than the step-sizes 
		allowed by 
		SDR 
		and by \cite[Theorem~3.1]{Jingjing20213OPSDR}.
	\end{enumerate}
\end{rem}
Next, we present the result of convergence for 
Algorithm~\ref{algo:SDR} whose proof is similar to the proof of 
Theorem~\ref{teo:MTC}.
\begin{teo} \label{teo:SDR}
	In the context of Problem~\ref{prob:main}, let $\gamma\in 
	\RPP$, 	let $(x_{-1},x_{0},y_{0})\in V\times 
	V\times 
	V^{\bot}$, 
	let 
	$w_{-1} = Bx_{-1}$,
	and let 
	$(x_n)_{n 
		\in 
		\N}$ and 
	$(y_n)_{n \in \N}$ be the sequences generated by
	Algorithm~\ref{algo:SDR}.
	Suppose that $\gamma$ satisfies,
	\begin{equation}\label{eq:gammaSDR}
		\frac{2}{3}-(2\beta+\zeta)\gamma-\beta^2\zeta\gamma^3 > 0.
	\end{equation}
	Then,  there exist
	$\overline{x}\in \zer(A+B+C+N_V)$ and $\overline{y}\in 
	V^{\bot}\cap(A\overline{x}+P_{V}(B+C)\overline{x})$
	such that $x_{n}\weak \overline{x}$ and $y_{n}\weak 
	\overline{y}$.
\end{teo}
\begin{proof}
	By proceeding similarly to the proof of 
	Theorem~\ref{teo:MTC}, we 
	have that the sequence $(z_n)_{n \in \N}$, given by 
	$z_n=x_n+\lambda y_n$, corresponds to the sequence generated 
	by 
	\eqref{eq:SDRCOR2} for $\lambda = 1$ and the operators 
	defined in 
	\eqref{eq:defOp}. Therefore, by noticing that $\lambda=1$ 
	satisfies,
	\begin{equation}
		\frac{2}{3}-(2(\beta\gamma)+(\zeta\gamma))\lambda-(\beta\gamma)^2(\zeta\gamma)\cdot
		\lambda^3 > 0,
	\end{equation}
	the result follows by Proposition~\ref{prop:Jingjing}.
\end{proof}
\begin{rem}\label{rem:2}
	\begin{enumerate}
		\item Note that, Algorithm~\ref{algo:SDR} has the same 
		advantages mentioned in Remark~\ref{rem:1}.\ref{rem:11} 
		over 
		the method proposed in 
		\cite{BricenoRoldanYuchaoChen}.
		\item\label{rem:22}  Similarly to 
		Remark~\ref{rem:1}.\ref{rem:11}, by considering an 
		additional 
		parameter $\lambda \in \RPP$, we can extend 
		Algorithm~\ref{algo:SDR} to the 
		following routine  
		\begin{equation}
			\label{eq:algoCMT}
			(\forall\, n \in \N) \quad	\left\lfloor
			\begin{aligned}
				%&\text{for }n=1,2,\ldots\\
				&w_{n}=B x_n,\\
				&\text{find }\, (p_{n},q_{n})\in \mathcal{H}^2\, 
				\text{ 
					such that:}\\
				&p_{n}+\gamma 
				q_{n}=x_{n}+\gamma y_{n}-\lambda P_V (w_n+Cx_n)  
				\text{ 
					and } 
				\\\
				& \quad \quad  \quad  \quad 
				\frac{P_{V}q_{n}}{\lambda}+P_{V^{\bot}}q_{n}\in 
				A\left(P_{V}p_{n}+\frac{P_{V^{\bot}}p_{n}}{\lambda}\right),\\
				&x_{n+1}=P_{V}(p_{n}-\lambda(w_n-w_{n-1})),\\
				&y_{n+1}=P_{V^{\bot}}q_{n}.
			\end{aligned}
			\right.
		\end{equation}
		In this case, the weak convergence of $(x_n)_{n \in \N}$ 
		to a point 
		in $Z$, is guaranteed for $\gamma \in \RPP$ and 
		$(\lambda\gamma)$ 
		satisfying 
		\eqref{eq:gammaBSDR}.
	\end{enumerate}
\end{rem}

\section{Applications}\label{sec:app}
\subsection{Extension to multivariate monotone inclusions}
In this section, we aim to solve the following composite 
primal-dual monotone inclusion.
\begin{pro}\label{pro:multi}
	Let $(I,J)\in \N^2$, set $\mathcal{I}=\{1,\ldots,I\}$ and $\mathcal{J}=\{1,\ldots,J\}$. For every $i \in \mathcal{I}$ and every $j \in \mathcal{J}$, let $\H_i$ and $\G_j$ be real Hilbert spaces. Set $\bm{\H} = \oplus_{i\in \mathcal{I}} \H_i$ and $\bm{\G} = \oplus_{j \in \mathcal{J}} \G_j$. For every $i \in \mathcal{I}$ and every $j \in \mathcal{J}$, let $A_i\colon \H_i\to 2^{\H_i}$ and $M_{j} \colon \G_j \to 2^{\G_j}$ be maximally monotone operators, let $L_{i,j} \colon \H_i \to \G_j$ be a bounded linear operator, let $C_i \colon \H_i \to \H_i$ be $\zeta_i^{-1}$-cocoercive for $\zeta_i \in \RPP$, let $\bm{B} \colon \bm{\H} \to \bm{\H} \colon \x \mapsto (B_i\x)_{i \in I}$, be $\tilde{\beta}$-Lipschitzian for $\tilde{\beta}\in \RPP$, let $N_j \colon \G_j \to \G_j$ be monotone such that $N_j^{-1}$ is $\nu_i$-Lipschitzian, for $\nu_j \in \RPP$,  let $D_j \colon \G_j \to \G_j$ be maximally monotone and $\delta_{j}$-strongly monotone, for $\delta_{i} \in \RPP$, and let $V_i$ be a closed subspace of $\H_i$. We aim to solve the system of primal inclusions
	\begin{align}\label{eq:multipri}
		&\text{ find } \x \in \bm{\H} \text{ such that} \nonumber\\
		&(\forall i \in \mathcal{I}) \quad  0 \in A_i x_i +\sum_{j \in \mathcal{J}} L^*_{i,j}(M_j\infconv N_j  \infconv D_j) \left(\sum_{i \in I}L_{i,j}x_i\right)  + B_i \x + C_ix_i + N_{V_i}x_i ,  
	\end{align} 
	together with the associated system of dual inclusions
	\begin{align}\label{eq:multidual}
		\text{ find }& \bm{u} \in  \bm{\G}
		\text{ such that} \nonumber \\
		& (\exists \x \in \bm{\H})\ (\forall i \in \mathcal{I})\ (\forall j \in \mathcal{J})\quad  
		\begin{cases}
			-\sum_{j \in \mathcal{J}} L^*_{i,j}u_{j} \in A_i x_i + B_i x_i + C_i\x+N_{V_i}x_i,\\
			u_{j} \in (M_j\infconv N_j  \infconv D_j)\left(\sum_{i \in \mathcal{I}}L_{i,j}x_i\right),
		\end{cases}
	\end{align}
	under the assumption that the primal-dual solution set $\widetilde{\bm{Z}}$ is not empty.
\end{pro}
This multivariate inclusion problem has been studied in \cite{AttouchBricenoCombettes2010,CombettesMinh2022,Comb13,CombettesEckstein2018MP} in the case when, for every $i \in \mathcal{I}$, $V_i=\H_i$. As a consequence of Theorem~\ref{teo:MTC} and Theorem~\ref{teo:SDR}, we propose two algorithms for solving Problem~\ref{pro:multi} requiring, for each $i \in I$ and each $k \in K$, only one activation of $L_{i,k}$ and $L^*_{i,k}$ at each iteration.
\begin{cor}\label{cor:MTCMulti}
	In the context of 
	Problem~\ref{pro:multi}, for every $i \in \mathcal{I}$, 
	let $(x_{-1}^{1,i},x_{0}^{1,i},y_{0}^{1,i})\in V_i\times V_i\times V_i^{\bot}$, for every $j \in \mathcal{J}$, 
	let $(x_{-1}^{2,j},x_{0}^{2,j})\in {\G}_j \times {\G}_j$, and define $\bm{x}_{-1}^1 = (x_{-1}^{1,i})_{i \in \mathcal{I} }$,  $w_{-1}^{1,i} = B_i \x_{-1}^1-\sum_{j \in \mathcal{J}}L^*_{i,j}x_{-1}^{2,j}$, and $w_{-1}^{2,j} = 
	N_j^{-1} x_{-1}^{2,j}+\sum_{i \in \mathcal{I}}L_{i,j}x_{-1}^{1,i}$. Moreover, set $\ell = \sum_{k\in K}\left(\sum_{i \in I}\|L_{i,k}\|\right)^2$, $\beta = \max\{\widetilde{\beta},\nu_1,\ldots,\nu_J\} + \sqrt{\ell}$, $\zeta = 
	\max\{\zeta_1,\ldots,\zeta_I, \delta_1,\ldots,\delta_J\}$,  and  let 
	\begin{equation}\label{eq:gammaMulti}
		\gamma \in 	
		\left]0,\frac{2}{4\beta+\zeta}\right[.			
	\end{equation}
	Consider the sequences $(\x_n^1)_{n \in \N}=(x_n^{1,1},\ldots,x_n^{1,I})_{n \in \N}$ and $(\x_n^2)_{n \in \N}=(x_n^{2,1},\ldots,x_n^{2,J})_{n \in \N}$ defined recursively by
	\begin{equation}
		\label{eq:algoMTCMulti}
		(\forall n \in \N)\quad \left\lfloor
		\begin{aligned}
			&\text{for } i=1,\ldots,I\\
			&\left\lfloor
			\begin{aligned}
				&w_n^{1,i} = B_i \x_n^1-\sum_{j \in \mathcal{J}}L^*_{i,j}x_n^{2,j}\\
				&u_n^{1,i} = x_n^{1,i} +\gamma y_n^{1,i} - \gamma P_{V_i} (2w_n^{1,i}-w_{n-1}^{1,i}+C_ix_n^{1,i})\\
				&p_n^{1,i} = J_{\gamma A_i} u_n^{1,i}\\
				&x_{n+1}^{1,i} = P_{V_i} p_n^{1,i}\\
				&y_{n+1}^{1,i}=y_{n}^{1,i}-\frac{p_n^{1,i}-x_{n+1}^{1,i}}{\gamma}.
			\end{aligned}\right.\\
			&\text{for } j=1,\ldots,J\\
			&\left\lfloor
			\begin{aligned}
				&w_n^{2,j} = N_j^{-1} x_n^{2,j}+\sum_{i \in \mathcal{I}}L_{i,j}x_n^{1,i}\\
				&u_n^{2,j} = x_n^{2,j}  - \gamma  (2w_n^{2,j}-w_{n-1}^{2,j}+D_j^{-1}x_n^{2,j})\\
				&x_{n+1}^{2,j} = J_{\gamma M^{-1}_j} u_n^{2,j}
			\end{aligned}\right.
		\end{aligned}
		\right.
	\end{equation}
	Therefore, $(\bm{x}_n^1)_{n \in \N}$ and $(\bm{x}_n^2)_{n \in \N}$ converges weakly to a solution to \eqref{eq:multipri} and \eqref{eq:multidual}, respectively.
\end{cor}

\begin{proof}
	Set $\mathbb{H}=\HH\oplus \bm{\G} $ and define the operators
	\begin{equation}
		\displaystyle
		\begin{cases}
			\bm{A} \colon  \mathbb{H} \to 2^{\mathbb{H}} \colon (\x,\bm{u})\mapsto  A_1x_1\times \ldots A_I x_I \times M_1^{-1}u_1 \times \ldots M_J^{-1}u_J \\
			\bm{B}_1 \colon  \HH \to \HH\colon \x \mapsto (B_1\x,\ldots,B_I\x)\\
			\bm{B}_2 \colon  \bm{\G} \to \bm{\G} \colon \bm{u} \mapsto (N^{-1}_1 u_1,\ldots,N^{-1}_Ju_J)\\
			\bm{L} \colon  \HH \to \HH \colon \x \mapsto (\sum_{i\in \mathcal{I}}L_{i,1}x_i,\ldots,\sum_{i\in \mathcal{I}}L_{i,J}x_i)\\
			%\bm{L}_2 \colon  \bm{\G} \to  \bm{\G} \colon \bm{u} \mapsto (\sum_{j \in \mathcal{J}}L_{1,j}^*u_j,\ldots,\sum_{j \in \mathcal{J}}L_{I,j}^*u_j)\\	
			%(\forall i \in \mathcal{I}) \quad  \bm{B}_i \colon  \H_i\times \bm{\G}\colon (x_i,\bm{u}) \mapsto M_ix_i+\sum_{j \in \mathcal{J}}L_{1,j}^*u_j\\
			%(\forall j \in \mathcal{J}) \quad  \bm{B}_{I+j} \colon  \bm{\H} \times \G_j \colon (\x,u_j) \mapsto N_j^{-1}u_j-\sum_{i\in \mathcal{I}}L_{i,J}x_i\\
			\bm{B} \colon  \mathbb{H} \to \mathbb{H} \colon  (\x,\bm{u}) \mapsto  (\bm{B}_1\bm{x}+\bm{L}^*\bm{u})\times (\bm{B}_2\bm{u}-\bm{L}\x)\\
			\bm{C}\colon   \mathbb{H} \to \mathbb{H} \colon (\bm{x},\bm{u}) \mapsto (C_1 \x,\ldots,C_I \x,D_1^{-1}u_1,\ldots,D_J^{-1}u_m)\\
			\bm{V} = V_1\times \ldots \times V_I \times \G_1 \ldots \times \G_J
		\end{cases}
	\end{equation}
	Note that 
	\begin{equation*}
		N_{\bm{V}} = N_{V_1}\times \ldots \times N_{V_I} \times \G_1\times \ldots \times \G_J
	\end{equation*}
	Hence, the primal-dual problem associated to \eqref{eq:multipri} and \eqref{eq:multidual} can be written as
	\begin{equation}
		\text{ find } \bm{z}=(\x,\bm{u}) \in \mathbb{H} 
		\text{ such that } \  0 \in \bm{A}\bm{z} + \bm{B}\bm{z}  +\bm{C}\bm{z} +N_{\bm{V}}\bm{z} .
	\end{equation}
	Additionally, it follows from \cite[Proposition~23.18 \& Proposition~29.3]{bauschkebook2017} that
	\begin{align*}
		&J_{\gamma \bm{A}}(\bm{x},\bm{u}) =(J_{\gamma A_1}x_1,\ldots,J_{\gamma A_I}x_I,J_{\gamma M^{-1}_1}u_1,\ldots ,J_{\gamma M^{-1}_J}u_J),\\
		&P_{\bm{V}}(\bm{x},\bm{u}) = (P_{V_1}x_1,\ldots,P_{V_I}x_I,u_1,\ldots,u_J). 
	\end{align*}
	Moreover,
	\begin{equation*}
		\bm{L}^* \colon  \bm{\G} \to \bm{\H} \colon \bm{u} \mapsto\left(\sum_{j \in \mathcal{J}}L_{1,j}^*u_j,\ldots,\sum_{j \in \mathcal{J}}L_{I,j}^*u_j\right)
	\end{equation*}
	and
	\begin{equation}
		(\forall \x \in \HH) \quad \|\bm{L}\x\|^2= \sum_{k\in K} \left(\sum_{i \in I}\|L_{i,k}x_i\|\right)^2\leq  \sum_{k\in K}\left(\sum_{i \in I}\|L_{i,k}\|\right)^2\|\x\|^2 = \ell\|\x\|^2.
	\end{equation}
	Hence, $\|L\| \leq \sqrt{\ell}$. Now, define 
	\begin{equation}
		(\forall n \in \N) \quad \begin{cases}
			\bm{z}_n = (x_n^{1,1},\ldots,x_n^{1,I},x_n^{2,1},\ldots,x_n^{1,j})\\
			\bm{w}_n = (w_n^{1,1},\ldots,w_n^{1,I},w_n^{2,1},\ldots,w_n^{1,j})\\
			\bm{u}_n = (u_n^{1,1},\ldots,u_n^{1,I},u_n^{2,1},\ldots,u_n^{1,j})\\
			\bm{p}_n = (p_n^{1,1},\ldots,p_n^{1,I},0,\ldots,0)\\
			\bm{y}_n = (y_n^{1,1},\ldots,y_n^{1,I},0,\ldots,0)
		\end{cases}
	\end{equation}
	Therefore, it follows from \eqref{eq:algoMTCMulti} that
	\begin{equation}
		\label{eq:algoMTCMULTI}
		(\forall n \in \N)\quad \left\lfloor
		\begin{aligned}
			&\bm{w}_n = \bm{B} \bm{x}_n\\
			&\bm{u}_n = \bm{x}_{n}+\gamma 
			\bm{y}_{n}-\gamma P_{\bm{V}} (2 \bm{w}_n - \bm{w}_{n-1}+\bm{C}\bm{x}_n)\\
			&\bm{p}_{n}=J_{{\gamma}\bm{A}} \bm{u}_n\\
			&\bm{x}_{n+1}=P_{\bm{V}}\bm{p}_n\\
			&\bm{y}_{n+1}=\bm{y}_{n}-\frac{\bm{p}_n-\bm{x}_{n+1}}{\gamma}.
		\end{aligned}
		\right.
	\end{equation}
	Since $\bm{A}$ is maximally monotone \cite[Proposition~20.23]{bauschkebook2017},  $\bm{B}$ is monotone and $\beta$-Lipschitzian
	\cite[eq.(3.11)]{CombPes12}, and $\bm{C}$ is $\zeta^{-1}$-cocoercive \cite[eq.(3.12)]{Vu13}, the result follows from Theorem~\ref{teo:MTC}.
\end{proof}
The next result is as a consequence of Theorem~\ref{teo:SDR}.
\begin{cor}\label{cor:SDRMulti}
	In the context of 
	Problem~\ref{pro:multi}, for every $i \in \mathcal{I}$, 
	let $(x_{-1}^{1,i},x_{0}^{1,i},y_{0}^{1,i})\in V_i\times V_i\times V_i^{\bot}$, for every $j \in \mathcal{J}$, 
	let $(x_{-1}^{2,j},x_{0}^{2,j})\in {\G}_j \times {\G}_j$, and define $\bm{x}_{-1}^1 = (x_{-1}^{1,i})_{i \in \mathcal{I} }$,  $w_{-1}^{1,i} = B_i \x_{-1}^1-\sum_{j \in \mathcal{J}}L^*_{i,j}x_{-1}^{2,j}$, and $w_{-1}^{2,j} = 
	N_j^{-1} x_{-1}^{2,j}+\sum_{i \in \mathcal{I}}L_{i,j}x_{-1}^{1,i}$. Moreover, set $\ell = \sum_{k\in K}\left(\sum_{i \in I}\|L_{i,k}\|\right)^2$, $\beta = \max\{\widetilde{\beta},\nu_1,\ldots,\nu_J\} + \sqrt{\ell}$, $\zeta = 
	\max\{\zeta_1,\ldots,\zeta_I, \delta_1,\ldots,\delta_J\}$,  and  let $\gamma \in \RPP$ be such that
	\begin{equation}\label{eq:gammaSDRMulti}
		\frac{2}{3}-(2\beta+\zeta)\gamma-\beta^2\zeta\gamma^3 > 0.
	\end{equation}
	Consider the sequences $(\x_n^1)_{n \in \N}=(x_n^{1,1},\ldots,x_n^{1,I})_{n \in \N}$ and $(\x_n^2)_{n \in \N}=(x_n^{2,1},\ldots,x_n^{2,J})_{n \in \N}$ defined recursively by
	\begin{equation}
		\label{eq:algoSDRMulti}
		(\forall n \in \N)\quad \left\lfloor
		\begin{aligned}
			&\text{for } i=1,\ldots,I\\
			&\left\lfloor
			\begin{aligned}
				&w_n^{1,i} = B_i \x_n^1-\sum_{j \in \mathcal{J}}L^*_{i,j}x_n^{2,j}\\
				&u_n^{1,i} = x_n^{1,i} +\gamma y_n^{1,i} - \gamma P_{V_i} (w_n^{1,i}+C_ix_n^{1,i})\\
				&p_n^{1,i} = J_{\gamma A_i} u_n^{1,i}\\
				&x_{n+1}^{1,i} = P_{V_i} (p_n^{1,i}-\gamma (w_n^{1,i}-w_{n-1}^{1,i}))\\
				&y_{n+1}^{1,i}=y_{n}^{1,i}-\frac{p_n^{1,i}-x_{n+1}^{1,i}}{\gamma}.
			\end{aligned}\right.\\
			&\text{for } j=1,\ldots,J\\
			&\left\lfloor
			\begin{aligned}
				&w_n^{2,j} = N_j^{-1} x_n^{2,j}+\sum_{i \in \mathcal{I}}L_{i,j}x_n^{1,i}\\
				&u_n^{2,j} = x_n^{2,j}  - \gamma  (w_n^{2,j}+D_j^{-1}x_n^{2,j})\\
				&x_{n+1}^{2,j} = J_{\gamma M^{-1}_j} u_n^{2,j}-\gamma(w_n^{2,j}-w_{n-1}^{2,j}).
			\end{aligned}\right.
		\end{aligned}
		\right.
	\end{equation}
	Therefore, $(\bm{x}_n^1)_{n \in \N}$ and $(\bm{x}_n^2)_{n \in \N}$ converges weakly to a solution to \eqref{eq:multipri} and \eqref{eq:multidual}, respectively.
\end{cor}
\begin{proof}
	Analogous to the proof of Corollary~\ref{cor:MTCMulti}.
\end{proof}
\begin{rem}
	\begin{enumerate}
		\item Recurrences generated by 
		\eqref{eq:algoMTCMulti} and \eqref{eq:algoSDRMulti} have the 
		same advantages mentioned in Remark~\ref{rem:1} in 
		comparison 
		to the 
		method proposed in 
		\cite[Section~4]{BricenoRoldanYuchaoChen}.
	\end{enumerate}	
\end{rem}
Next, we proceed to apply our methods to the example 
of convex optimization problems available in 
\cite{BricenoRoldanYuchaoChen}.
\begin{ejem}\label{ex:1}\cite[Example 
	4.4]{BricenoRoldanYuchaoChen}
	Consider the convex minimization problem
	\begin{equation}\label{eq:probopti1}
		\min_{\sx\in\sV}\big(\mathsf{f}(\sx)+\mathsf{h}(\sx)+\sg(\sL\sx)\big),
	\end{equation}
	where $\mathsf{f}\in\Gamma_0(\sH)$, 
	$\mathsf{g}\in\Gamma_0(\sG)$, and 
	$\mathsf{h}\colon\sH\to\R$ is 
	convex differentiable with $\rho^{-1}$-Lipschitzian gradient. 
	Note 
	that, under the qualification condition 
	\begin{equation}
		\label{eq:condqual}
		0 \in\sri\Big(\mathsf{L}(\sV\cap\dom 
		\mathsf{f})-\dom\mathsf{g}\Big),
	\end{equation}
	by \cite[Theorem~16.47 \& Example~16.13]{bauschkebook2017}, 
	the problem 
	in 
	\eqref{eq:probopti1} is equivalent to
	\begin{equation}
		\text{ find } \sx \in \sH \text{ such that } 0 \in 
		\partial 
		\mathsf{f}
		(\sx) + \sL^*\partial \sg (\sL \sx) + \nabla \mathsf{h}
		(\sx)+N_{\sV}{\sx}.	
	\end{equation}
	Therefore, by defining $A=\partial \mathsf{f}$, $M=  
	\partial 
	\mathsf{g}$, $C=\nabla 
	\mathsf{h}$, and $B=D^{-1}=N^{-1}=0$, the optimization 
	problem 
	in 
	\eqref{eq:probopti1} is equivalent 
	to the primal inclusion in \eqref{eq:multipri} when $I=1$ and $J=1$. Hence, the 
	recurrence in
	\eqref{eq:algoMTCMulti}, can be applied to this setting, deriving 
	the 
	following algorithm:
	\begin{equation}\label{eq:algoMTOP1}
		(\forall n \in \N)\quad \left\lfloor
		\begin{aligned}
			&\sw_n^1 = \sL^* \su_n\\
			&\sw_n^2 = - \sL \sx_n\\
			&\sp_{n}=\prox_{\gamma f}\big(\sx_{n}+\gamma 
			\sy_{n}-\gamma P_{\sV} (2 \sw^1_n - 
			\sw^1_{n-1}+\nabla 	\mathsf{h} \sx_n)\big)\\
			&\su_{n+1}=\prox_{\gamma g^*}\big(\su_{n}-\gamma(2 
			\sw^2_n 
			- \sw^2_{n-1})\big)\\
			&\sx_{n+1}=P_{\sV}\sp_n\\
			&\sy_{n+1}=\sy_{n}-\frac{\sp_n-\sx_{n+1}}{\gamma},
		\end{aligned}
		\right.
	\end{equation}
	where the weak convergence of $(\sx_n)_{n \in \N}$ to a 
	solution 
	to 
	\eqref{eq:probopti1} is guaranteed for $\gamma \in ]0, 
	2/(4\|L\|+\rho)[$.
	Additionally, if we apply the 
	recurrence in
	\eqref{eq:algoSDRMulti} on this setting, we obtain the following 
	algorithm
	\begin{equation}\label{eq:algoSDROP1}
		(\forall n \in \N)\quad \left\lfloor
		\begin{aligned}
			&\sw_n^1 = \sL^* \su_n\\
			&\sw_n^2 = - \sL \sx_n\\
			&\sp_{n}=\prox_{\gamma f}\big(\sx_{n}+\gamma 
			\sy_{n}-\gamma P_{\sV} ( \sw^1_n+\nabla 	
			\mathsf{h} 
			\sx_n)\big)\\
			&\su_{n+1}=\prox_{\gamma g^*}\big(\su_{n}-\gamma
			\sw^2_n \big)-\gamma(\sw^2_n 
			- \sw^2_{n-1})\\
			&\sx_{n+1}=P_{\sV}(\sp_n-\gamma(\sw^1_n - 
			\sw^1_{n-1}))\\
			&\sy_{n+1}=\sy_{n}-\frac{\sp_n-\sx_{n+1}}{\gamma},
		\end{aligned}
		\right.
	\end{equation}
	where $(\sx_n)_{n \in \N}$ converges weakly to a 
	solution 
	to \eqref{eq:probopti1} if $\gamma$ satisfies
	\begin{equation}
		\frac{2}{3}-(2\|L\|+\rho)\gamma-\|L\|^2\rho\gamma^3 > 0.
	\end{equation}
\end{ejem}
\subsection{\bf Inclusion Involving a Sum of Maximally Monotone Operators}	
In this section we study the following problem involving a sum of maximally monotone operators.
\begin{pro}\label{prob:sumofMM}
	Let $K \in \N$, for every $k \in \{1,\ldots,K\}$, let $A_k: \H \to 2^\H$ be a maximally monotone operator, let
	$B : \H \to \H$ be a $\beta$-Lipschitzian operator for some 
	$\beta >0$, and let $C: \H \to \H$ be a $\zeta^{-1}$-cocoercive 
	operator 
	for some 
	$\zeta >0$. 
	The 
	problem is to 
	\begin{equation}\label{eq:sumofMM}
		\text{find} \quad x \in \H \quad \text{such that} \quad 0 
		\in 
		\sum_{k=1}^K A_k x+Bx+Cx,
	\end{equation}
	under the assumption that its solution set not empty.
\end{pro}
This problems arise for example in convex variational inverse problems \cite{CombPesquet2008,Raguet2013}, signal denoising \cite{CombePesquet2007,ROF1992}, economics \cite{Jofre2007equilibrium},  image restoration \cite{chambolle1997,Figueiredo2007}, among others.
In the case when $C=0$ this problem can be solved by the methods proposed in \cite{Briceno2015JOTA,CombPes12} and when $B=0$ by the method proposed in \cite{briceno2015Optim,Raguet2013,Vu13}. Problem~\ref{prob:sumofMM} can be solved by the method proposed in \cite[Section~5]{BricenoDavis2018}, however, this method involves restrictive step-sizes, an additional activation of the Lipschitzian operator in each iteration, and dual variables which could be computationally expensive in large dimensions. In this section we provide two methods for solving Problem~\ref{prob:sumofMM} that exploit the structure of Problem~\ref{prob:sumofMM} without incorporating dual variables. The first method is presented in the following corollary  which is obtained as a consequence of Theorem~\ref{teo:MTC}.
\begin{cor}\label{cor:MTsMM}
	In the context of Problem~\ref{prob:sumofMM}, let $\gamma\in 
	]0,2/(4\beta+\zeta)[$, let $\{\omega_1,\ldots,\omega_K\} \subset [0,1]$ be such that $\sum_{k=1}^K\omega_k=1$, let $(x_{-1},x_{0})\in \H^2$, set
	$w_{-1} = Bx_{-1}$, and let $(y_0^k)_{k=1}^K \in \H^K$ be such that $\sum_{k=1}^K\omega_k y_k=0$. Let
	$(x_n)_{n 
		\in 
		\N}$ be the sequence generated  the following recurrence
	\begin{equation}
		\label{eq:algoMTCsMM}
		(\forall n \in \N)\quad \left\lfloor
		\begin{aligned}
			&w_n = B x_n\\
			&q_n = x_n-\gamma(2w_n-w_{n-1}+Cx_n)\\
			&\text{for } k=1,\ldots,K\\
			&\left\lfloor
			\begin{aligned}
				&p_{n}^k=J_{{\gamma}A_k/\omega_k}  (\gamma y_n^k + q_n)
			\end{aligned}
			\right.\\
			&x_{n+1} = \sum_{k=1}^K\omega_k p_n^k\\
			&\text{for } k=1,\ldots,K\\
			&\left\lfloor
			\begin{aligned}
				&y_{n+1}^k=y_{n}^k-\frac{p_n^k-x_{n+1}}{\gamma}.
			\end{aligned}
			\right.
		\end{aligned}
		\right.
	\end{equation}
	Then,  there exist
	$\overline{x}\in \zer(\sum_{k=1}^K A_k+B+C)$
	such that $x_{n}\weak \overline{x}$.
\end{cor}
\begin{proof}
	Set $\HH = (\H^K,\scal{\cdot}{\cdot}_\omega)$, where, for every $\x= (x_k)_{k=1}^K\in \H^K$ and $\y= (y_k)_{k=1}^K \in \H^K$, $\scal{\x}{\y}_\omega = \sum_{k=1}^{K}\omega_k\scal{x_k}{y_k}$. Define the following operators
	\begin{align*}
		&\bm{A} \colon \HH \to 2^{\HH} \colon \x=(x_k)_{k=1}^K \mapsto \frac{1}{\omega_1}A_1 x_1 \times \frac{1}{\omega_2}A_2 x_2 \times \ldots\times \frac{1}{\omega_k}A_K x_K\\
		&\bm{B} \colon \HH \to \HH \colon \x=(x_k)_{k=1}^K \mapsto (B x_1,  \ldots, B x_K)\\
		&\bm{C} \colon \HH \to \HH \colon \bm{x}=(x_k)_{k=1}^K \mapsto (C x_1, \ldots, C x_K)\\
		&\bm{J} \colon \H \to \HH \colon x \mapsto (x, \ldots, x).
	\end{align*}
	Therefore, it follows from \cite[Proposition~4.1]{Briceno2015JOTA} that $\bm{A}$ is maximally monotone, $\bm{B}$ is $\beta$-Lipschitizian,
	and from \cite[Proposition~6.2]{briceno2015Optim} that $\bm{C}$ is $\zeta^{-1}$-cocoercive.
	Now, define
	\begin{align*}
		V = \menge{\x=(x_k)_{k=1}^K \in \HH}{x_1=x_2=\ldots=x_K}.
	\end{align*}
	Hence, it follows from \cite[Proposition~4.1]{Briceno2015JOTA} that $V$ is a closed vector subspace of $\HH$ and that, for every $\x=(x_k)_{k=1}^K \in \HH$, $P_V \x = \bm{J}(\sum_{k=1}^K\omega_k x_k)$ and 
	\begin{equation}
		N_V(\x) = \begin{cases}
			\menge{\y=(y_k)_{k=1}^K \in \HH}{\sum_{k=1}^K\omega_k y_k=0}, \text{ if } \x \in V\\
			\varnothing, \text{ otherwise}.
		\end{cases}
	\end{equation} 
	Then, for every $x \in \H$, we have
	\begin{align}\displaystyle
		0 \in \sum_{k=1}^K A_kx+Bx+Cx &\Leftrightarrow (\exists  \y = (y_k)_{k=1}^K \in \bm{A} (\bm{J}x)) \quad  0 = \sum_{k=1}^N \omega_k (y_k + Bx+Cx)\nonumber\\
		&\Leftrightarrow  (\exists  \y = (y_k)_{k=1}^K \in \bm{A} (\bm{J}x)) \quad -\left(\y+\bm{J}((B+C)x)\right) \in N_V(\bm{J}(x))\nonumber\\
		&\Leftrightarrow 0 \in (\bm{A}+\bm{B}+\bm{C}+N_V)(\bm{J}(x))\nonumber\\
		&\Leftrightarrow \bm{J}(x) \in \zer(\bm{A}+\bm{B}+\bm{C}+N_V).\label{eq:equivsol}
	\end{align}
	Now, define, for every $n \in \N$, $\x_n=\bm{J}(x_n)$, $\bm{w}_n=\bm{J}(w_n)$, $\bm{u}_n = (x_n^k+\gamma y_n^k - \gamma q_n)_{k=1}^K$, $\bm{p}_n = (p_n^k)_{k=1}^K$, and $\bm{y}_n = (y_n^k)_{k=1}^K $. Therefore, it follows from \eqref{eq:algoMTCsMM} that
	\begin{equation}
		\label{eq:algoMTCMulti2}
		(\forall n \in \N)\quad \left\lfloor
		\begin{aligned}
			&\bm{w}_n = \bm{B} \x_n\\
			&\bm{u}_n = \x_{n}+\gamma 
			\y_{n}-\gamma P_V (2 \bm{w}_n - \bm{w}_{n-1}+\bm{C}\x_n)\\
			&\bm{p}_{n}=J_{{\gamma}\bm{A}} \bm{u}_n\\
			&\x_{n+1}=P_V\bm{p}_n\\
			&\y_{n+1}=\y_{n}-\frac{\bm{p}_n-\x_{n+1}}{\gamma}.
		\end{aligned}
		\right.
	\end{equation}
	In view of Theorem~\ref{teo:MTC}, we conclude that $\x_n\weak \x \in \zer(\bm{A}+\bm{B}+\bm{C}+N_V)$. Moreover, since $\x_n = \bm{J}(x_n)$, $x_n\weak x \in \H$ and $\bm{J}(x) = \x \in \zer(\bm{A}+\bm{B}+\bm{C}+N_V)$. The result follows from \eqref{eq:equivsol}.
\end{proof}
The next result is as a consequence of Theorem~\ref{teo:MTC}.
\begin{cor}
	In the context of Problem~\ref{prob:sumofMM}, let $\gamma>0$ be such that
	\begin{equation}
		\frac{2}{3}-(2\beta+\zeta)\gamma-\beta^2\zeta\gamma^3 > 0,
	\end{equation}
	let $\{\omega_1,\ldots,\omega_K\} \subset [0,1]$ be such that $\sum_{k=1}^K\omega_k=1$, let $(x_{-1},x_{0})\in \H^2$, set
	$w_{-1} = Bx_{-1}$, and let $(y_0^k)_{k=1}^K \in \H^K$ be such that $\sum_{k=1}^K\omega_k y_k=0$. Let
	$(x_n)_{n 
		\in 
		\N}$ be the sequence generated  the following recurrence
	\begin{equation}
		\label{eq:algoSDRsMM}
		(\forall n \in \N)\quad \left\lfloor
		\begin{aligned}
			&w_n = B x_n\\
			&q_n = x_n-\gamma(w_n+Cx_n)\\
			&\text{for } k=1,\ldots,K\\
			&\left\lfloor
			\begin{aligned}
				&p_{n}^k=J_{{\gamma}A_k/\omega_k}  (\gamma y_n^k + q_n)
			\end{aligned}
			\right.\\
			&x_{n+1} = \sum_{k=1}^K\omega_k p_n^k-\gamma(w_n-w_{n-1})\\
			&\text{for } k=1,\ldots,K\\
			&\left\lfloor
			\begin{aligned}
				&y_{n+1}^k=y_{n}^k-\frac{p_n^k-x_{n+1}}{\gamma}.
			\end{aligned}
			\right.
		\end{aligned}
		\right.
	\end{equation}
	Then,  there exist
	$\overline{x}\in \zer(\sum_{k=1}^K A_k+B+C)$
	such that $x_{n}\weak \overline{x}$.
\end{cor}
\begin{proof}
	Analogous to the proof of Corollary~\ref{cor:MTsMM}
\end{proof}
\section{Numerical experiments}\label{sec:numerical}
In \cite{BricenoRoldanYuchaoChen}, the authors
highlighted
the 
advantages of FPIHF of considering subspaces induced by kernels 
of 
linear 
operators allowing larger step-sizes than the
primal-dual 
algorithm (Condat-V\~u) \cite{Condat13,Vu13}, which is limited to 
small step-sizes. This implies 
that, in certain cases, FPIHF numerically outperforms the 
primal-dual 
algorithm. To compare 
the proposed methods with FPIHF and Condat-V\~u, 
we consider the following convex optimization problem proposed in 
\cite[Section~5]{BricenoRoldanYuchaoChen} and related to compute 
the fusion estimator in fused LASSO problems 
\cite{Friedman2007,Ohishi2021,Ohishi2022,Tibshirani2005}.

\begin{pro}\label{prob:numericproblem}
	Let $\eta^0=(\eta_i^0)_{1\le i\le N} \in \R^N $, let 
	$\eta^1=(\eta_i^1)_{1\le i\le N} \in \R^N$, let 
	$(\alpha_1,\alpha_2) \in \RPP^2$, let
	$M\in 
	\R^{K\times N}$, let $z \in \R^K$. Consider the following 
	optimization problem.
	\begin{equation}
		\label{eq:numericproblem}
		\min_{ \eta^0 \leq x\leq \eta^1}  
		\left(\frac{\alpha_1}{2}\|Mx-z\|^2+\alpha_2
		\|\nabla x\|_1\right),
	\end{equation}
	where $\nabla: \R^N \to \R^{N-1} : (\xi_i)_{1\leq i\leq N} 
	\mapsto 
	(\xi_{i+1}-\xi_i)_{1\leq i\leq N-1}$ is the discrete gradient.
\end{pro}
By 
\cite[Equation (5.2) \& Example 4.5]{BricenoRoldanYuchaoChen} 
Problem~\ref{prob:numericproblem} can be written equivalently 
as the optimization problem on Example~\ref{ex:1}, where $\sT\colon \sx = (\mathpzc{x},\mathpzc{w})\mapsto
	\mathpzc{A}\mathpzc{x}-\mathpzc{w}$, $
	\sV=\ker\sT$, $\nabla 
\mathsf{h}= 
\alpha_1(\id-z)$ is 
$\alpha_1-$Lipschitzian, $\mathsf{L}=\nabla$, and 
$\|\mathsf{L}\|=2$. Therefore, 
Problem~\ref{prob:numericproblem} can be solved by the methods 
in \eqref{eq:algoMTOP1} and \eqref{eq:algoSDROP1}, which reduces to 
Algorithm~\ref{algo:MTPD} and Algorithm~\ref{algo:CMTPD}. 
We 
recall that 
\cite[Example~24.22]{bauschkebook2017}
\begin{equation}
	\prox_{\alpha \|\cdot\|_1} \colon x \mapsto 
	\big(\text{sign}(x_i)\max\{|x_i|-\alpha,0\}\big)_{1\leq i 
		\leq n}
\end{equation}
and for $C=\menge{x \in \R^n}{\eta^0 \leq
	x\leq \eta^1}$ we have 
\begin{equation}
	\prox_{\iota_C} = P_C \colon x \mapsto 
	\max\{\eta_0,\min\{x,\eta_1\}\},
\end{equation}
where $\max$ and $\min$ are applied componentwise. Additionally, 
we 
use the identity in \eqref{e:Moreau_nonsme}.
\begin{algo} \label{algo:MTPD}
	Let $N=(\id+M^*M)^{-1}$, let 
	$\gamma = 0.999\cdot 2/(8+\alpha_1)$, let 	
	$x_0=w^1_{-1}=0_{\R^N}$, let 
	$z_0=0_{\R^K}$, and let
	$u_0=w^2_{-1}=0_{\R^{N-1}}$. Consider the 
	recurrence.
	\begin{equation}
		(\forall n \in \N)\quad \left\lfloor
		\begin{aligned}
			&w_n^{1} = \nabla^* u_n\\
			&w_n^{2}= -\nabla x_n\\
			&q_n^{1} =2 
			w^{1}_n - 
			w^{1}_{n-1}-M^*N(M(2
			w^{1}_n - 
			w^{1}_{n-1})+z_{n}-b)\\
			&q_n^{2} = z_n - 
			b+N(M(2 
			w^{1}_n - 
			w^{1}_{n-1})+z_{n}-b)\\
			&p^1_{n}=\max\{\eta_0,\min\{x_{n}+\gamma
			y^1_{n}-\gamma q_n^{1},\eta_1\}\}\\
			&p^2_{n}=z_{n}+\gamma 
			y^2_{n}-\gamma q_n^{2}\\
			&\widehat{u}_n = u_{n}/\gamma 
			-2 
			w_n^{2}
			+ w_{n-1}^{2}\\\
			&u_{n+1}=\gamma(\widehat{u}_n
			-\textnormal{sign}(\widehat{u}_n 
			)\cdot\max\{|\widehat{u}_n 
			|-\alpha/\gamma,0\})\\
			&x_{n+1}=N(p_n^1+M^*p_n^2)\\
			&z_{n+1}=M^*x_{n+1}\\
			&y^1_{n+1}=y^1_{n}-\frac{p^1_n-x_{n+1}}{\gamma}\\
			&y^2_{n+1}=y^2_{n}-\frac{p^2_n-z_{n+1}}{\gamma}.
		\end{aligned}
		\right.
	\end{equation}
\end{algo}
\begin{algo} \label{algo:CMTPD}
	Let 
	$\gamma = 0.999 \hat{\gamma}$ where $\hat{\gamma}$ satisfies
	$(2/3-(4+\alpha_1)\hat{\gamma}-16\alpha_1\hat{\gamma}^3) 
	=	0$, let 	
	$x_0=w^1_{-1}=0_{\R^N}$, let 
	$z_0=0_{\R^K}$, and let
	$u_0=w^2_{-1}=0_{\R^{N-1}}$. Consider the 
	recurrence.
	\begin{equation}
		(\forall n \in \N)\quad \left\lfloor
		\begin{aligned}
			&w_n^{1} = \nabla^* u_n\\
			&w_n^{2}= -\nabla x_n\\
			&q_n^{1} =2 
			w^{1}_n - 
			w^{1}_{n-1}-M^*N(Mw^{1}_n
			+z_{n}-b)\\
			&q_n^{2} = z_n - 
			b+N(Mw^{1}_n
			+z_{n}-b)\\
			&p^1_{n}=\max\{\eta_0,\min\{x_{n}+\gamma
			y^1_{n}-\gamma q_n^{1},\eta_1\}\}-\gamma
			(w^{1}_n - 
			w^{1}_{n-1})\\
			&p^2_{n}=z_{n}+\gamma 
			y^2_{n}-\gamma q_n^{2}-\gamma( 
			w^{1}_n - 
			w^{1}_{n-1})\\
			&\widehat{u}_n = u_{n}/\gamma 
			-
			w_n^{2}\\\
			&u_{n+1}=\gamma(\widehat{u}_n
			-\textnormal{sign}(\widehat{u}_n 
			)\cdot\max\{|\widehat{u}_n 
			|-\alpha/\gamma,0\})-\gamma(w_n^{2}
			-z w_{n-1}^{2})\\
			&x_{n+1}=N(p_n^1+M^*p_n^2)\\
			&z_{n+1}=M^*x_{n+1}\\
			&y^1_{n+1}=y^1_{n}-\frac{p^1_n-x_{n+1}}{\gamma}\\
			&y^2_{n+1}=y^2_{n}-\frac{p^2_n-z_{n+1}}{\gamma}.
		\end{aligned}
		\right.
	\end{equation}
\end{algo}
Note that, the 
function
$x \mapsto \frac{\alpha_1}{2}\|Mx-z\|^2$ is convex differentiable 
with gradient $\alpha_1 M^*(M\cdot-z)$, which gradient is  
$\alpha_1\|M\|^2$-Lipschitz. Therefore, 
Problem~\ref{prob:numericproblem} can 
be solved 
by the primal-dual algorithm proposed in 
\cite{Condat13,Vu13}. As we mentioned above, if $\|M\|$ is large, 
it 
is necessary to 
choose small step sizes for ensuring the convergence of the 
primal-dual algorithm (see 
\cite[Equation~4.14]{BricenoRoldanYuchaoChen}). Similarly, by 
using 
product space 
techniques, FHRB can be used to solve this problem by selecting 
step 
sizes from the interval 
$]0,2/(4\|L\|+\|M\|^2)[$. Consequently, if $\|M\|$ 
is large, it is necessary to choose a small value for 
$\gamma$. On this setting, FHRB reduces to the following 
algorithm.
\begin{algo}[FHRB]\label{algo:FHRB}
	Let $\gamma \in]0, 2/(8+\alpha_1\|M\|^2)[$, let 	
	$x_0=w^1_{-1}=0_{R^N}$, and let
	$u_0=w^2_{-1}=0_{R^{N-1}}$.
	\begin{equation}
		(\forall n \in \N)\quad \left\lfloor
		\begin{aligned}
			&w_n^{1} = L^* u_n\\
			&w_n^{2}= -L x_n\\
			&q_n = 2 w^{1}_n - 
			w^{1}_{n-1}+M^*(Mx_{n}-b)\\
			&x_{n+1}=\max\{\eta_0,\min\{q_n,\eta_1\}\}\\
			&\widehat{u}_n  = u_{n}/\gamma 
			-2 
			w_n^{2}
			+ w_{n-1}^{2}\\\
			&u_{n+1}=\gamma(\widehat{u}_n 
			-\textnormal{sign}(\widehat{u}_n 
			)\cdot\max\{|\widehat{u}_n 
			|-\alpha/\gamma,0\}).
		\end{aligned}
		\right.
	\end{equation}
\end{algo}
Routines of Condat-V\~u and FPIHF for solving 
Problem~\ref{prob:numericproblem} are available 
in \cite{BricenoRoldanYuchaoChen} in Algorithm~1 and Algorithm~2, 
respectively. 

The following numerical experiments were implemented in MATLAB 
2023A 
and run in a laptop with Windows 11, 2.1GHz, 8-core, AMD Ryzen 5 
3550H, 32 GB ram.
\subsection{Random Matrices}
In order to compare Condat-V\~u, FPIHF, FHRB, 
Algorithm~\ref{algo:MTPD}, and 
Algorithm~\ref{algo:CMTPD} we set, similarly to 
\cite{BricenoRoldanYuchaoChen},
$\alpha_1=5$ and 
$\alpha_2=0.5$ and we consider
$M=\kappa \cdot \text{rand}(N,K)$, 
$\eta^0=-1.5\cdot\text{rand}(N)$, 	
$\eta^1=1.5\cdot\text{rand}(N)$, 
$z=\text{randn}(N)$, and $\kappa \in \{1/10,1/20,1/30\}$
where $\text{rand}(\cdot,\cdot)$ and $\text{randn}(\cdot,\cdot)$ 
are 
functions in MATLAB generating matrices uniformly 
distributed. We consider $N \in\{400,800,1200\}$ 
and, 
for each $N$, we test three values of $K$ described in 
Tables~\ref{table:I}-\ref{table:III}. For each
instance
we generate
$20$ random realizations for $M$, 
$z$, $\eta^0$, and $\eta^1$.
As the authors note in \cite{BricenoRoldanYuchaoChen}, the 
average value of
$\|M\|$ is bigger when $\kappa$ is smaller, 
therefore, for 
smaller values of $\kappa$ more restrictive are the 
step-sizes for 
Condat-V\~u and FHRB. The step sizes described in 
Algorithm~\ref{algo:MTPD} and Algorithm~\ref{algo:CMTPD} are 
chosen 
as large as possible while ensuring their respective convergence. 
To 
find $\hat{\gamma}$ in Algorithm~\ref{algo:CMTPD}, we used the 
function solve from MATLAB. For Condat-V\~u and FHRB, we used the 
same step sizes as stated in \cite{BricenoRoldanYuchaoChen}.  As 
a 
stop criterion for our experiments, 
we 
consider the relative error with tolerance $10^{-6}$. 
Additionally, we include a limit of $50000$ iterations. If some 
method exceeds that limit, we write $\infty$ on the table. 
Tables~\ref{table:I}-\ref{table:III} shows the results of the 
experiments, exhibiting the average number of iterations and CPU 
time. 

From Tables~\ref{table:I}-\ref{table:III} we can note that, in 
all cases, Algorithm~\ref{algo:MTPD} outperforms FHRB in terms of 
both the average number of iterations and CPU time. This 
highlights the advantages of splitting the linear operator, which 
allows for larger step-sizes. Additionally, when $N=800$ or 
$N=1200$, 
FHRB reach 
the 
iteration limit while Algorithm~\ref{algo:MTPD} converges before 
reaching that limit. In all cases, for different 
values of 
$N$ and $K$, the step-sizes of Algorithm~\ref{algo:MTPD}  fall 
within the range of $]0, 2/(8+\alpha_1)[=]0, 0.15[$. On the other 
hand, the 
step-sizes of FHRB depends on $\|M\|$ and fall within 
the range of $]0, 
2/(8+\|M\|^2)[$. For instance, when $N=1200$, $K=1000$, 
and $\kappa =10$, the average value of $\|M\|$ is
$\overline{\|M\|}=54.77$, which yields 
$2/(8+\overline{\|M\|}^2) \approx 1.33 \cdot 10^{-4}$. 
This highlights the significant difference in admissible 
step-sizes between the two methods in this particular case.

Tables~\ref{table:I}-\ref{table:III} shows that 
Algorithm~\ref{algo:MTPD} and Algorithm~\ref{algo:CMTPD} 
outperform FPIHF in terms of the
average number of iterations and average CPU time. Furthermore, 
according 
to 
Table~\ref{table:IV}, the average numbers of 
iterations per
average time of iterations of Algorithm~\ref{algo:MTPD} and  
Algorithm~\ref{algo:CMTPD} are always 
higher than FPIHP. Consequently, each iteration of 
Algorithm~\ref{algo:MTPD} is 
less CPU time-consuming compared to FPIHP, providing empirical 
support for what was mentioned in 
Remark~\ref{rem:1}.\ref{rem:11}.

Comparing with Condat-V\~u, 
Tables~\ref{table:I}-\ref{table:III} indicates that 
Condat-V\~u needs considerably more iterations than 
Algorithm~\ref{algo:MTPD} and Algorithm~\ref{algo:CMTPD} to reach 
the 
stopping criterion. 
However, each iteration of Algorithm~\ref{algo:MTPD} and 
Algorithm~\ref{algo:MTPD} require more 
CPU time due to the additional calculations involved in the 
projections onto the subspace. From Tables~\ref{table:I} we 
conclude that, when $\kappa =1/10$, Algorithm~\ref{algo:CMTPD} 
has 
better performance than  Condat-V\~u. However, from 
Tables~\ref{table:II}-\ref{table:III}, we observe that 
Condat-V\~u 
outperforms Algorithm~\ref{algo:CMTPD} in several cases when 
$\kappa=1/20$ and when $\kappa=1/30$. On the other hand,
From Table~\ref{table:I}-\ref{table:III} we deduce
that Algorithm~\ref{algo:MTPD} outperforms Condat-V\~u in all 
cases 
excluding when
$\kappa=1/30$, $N=1200$, and $K=800$. In this particular case,  
Algorithm~\ref{algo:MTPD} is 
more competitive than FPIHF.
Note that Algorithm~\ref{algo:MTPD} outperforms
Condat-V\~u even when Condat-V\~u overcome FPIHF. 

From Table~\ref{table:IV} we deduce that 
Algorithm~\ref{algo:MTPD} 
and Algorithm~\ref{algo:CMTPD} have comparable CPU time per 
iteration.
However, from Tables~\ref{table:I}-\ref{table:III} we conclude 
that  
Algorithm~\ref{algo:MTPD} 
outperform Algorithm~\ref{algo:CMTPD} in average number of 
iterations 
and CPU time. This difference is attributable to the step-sizes 
of 
each algorithm. For  
Algorithm~\ref{algo:MTPD} the step-sizes is $\gamma \approx 
0.1538$, 
meanwhile, the step-sizes for Algorithm~\ref{algo:CMTPD} is 
$\gamma 
\approx 0.0732$. Note that, if we use the step-sizes provided by
\cite[Theorem~3.1]{Jingjing20213OPSDR} on 
Algorithm~\ref{algo:CMTPD}, 
they would be limited by $2/(18+3\cdot 
\alpha_1)=0.\overline{06}$.
{\small
	\begin{table}
		\caption{Comparison of Condat-V\~u, FPIHF, FHRB,  
			Algorithm~\ref{algo:MTPD}, and  
			Algorithm~\ref{algo:CMTPD} for 
			$\kappa=1/10$.\label{table:I}}
		\begin{tabular}{cccccccccc} 
			& 	
			
			&  Av. time (s)& Av. iter & Av. 
			time 
			(s)& 
			Av. iter & 
			Av. time (s)& Av. iter\\ 
			\cmidrule[0.5pt](lr){3-8}
			$N$& Algorithm & \multicolumn{2}{c}{$ 
				K=200$}&\multicolumn{2}{c}{$K=250$} 
			&\multicolumn{2}{c}{ 
				$K=300$}\\ 	\cmidrule[0.5pt](lr){1-1} 
			\cmidrule[0.5pt](lr){2-2}
			\cmidrule[0.5pt](lr){3-4} 
			\cmidrule[0.5pt](lr){5-6}\cmidrule[0.5pt](lr){7-8}
			\multirow{5}{*}{$400$} &	
			Condat-V\~u & 1.30 & 10103  & 1.24  & 10027  & 
			1.17 & 
			9383  \\ 
			& FPIHF  & 0.72 &  1518 & 1.41 & 2725 & 2.19 & 
			3451	
			\\ 
			& FHRB & 
			4.85 & 40546 & 4.45 & 36109 & 3.98 & 35526
			\\	 
			& Algorithm~\ref{algo:MTPD} & 0.29 & 925 & 0.58 & 
			1679 & 0.68 & 1794 \\
			& Algorithm~\ref{algo:CMTPD} & 0.56 & 1984 & 0.98 
			& 
			2998 & 1.51 & 3741 \\ \cmidrule[0.5pt](lr){3-8}
			\multicolumn{2}{c}{} & \multicolumn{2}{c}{$ 
				K=400$}&\multicolumn{2}{c}{$K=500$} 
			&\multicolumn{2}{c}{ 
				$K=600$}\\  
			\cmidrule[0.5pt](lr){3-4} 
			\cmidrule[0.5pt](lr){5-6}\cmidrule[0.5pt](lr){7-8}
			\multirow{5}{*}{$800$} &	Condat-V\~u  &
			10.78 & 24159  & 11.55 & 23255 & 10.65 & 21214 & 
			\\ 
			& FPIHF & 3.19 & 2823 & 6.36 & 3145 & 6.22 & 
			2283 	
			\\ 
			& FHRB & 
			$\infty$  
			& $\infty$ & $\infty$ & $\infty$ & $\infty$  & 
			$\infty$ \\	 
			& Algorithm~\ref{algo:MTPD} & 1.91 & 2241 & 2.76 
			& 
			1999 & 3.40 & 1941 \\					& 
			Algorithm~\ref{algo:CMTPD} & 2.49 & 3234 & 4.75 
			& 
			3339 & 4.77 & 2669 \\ \cmidrule[0.5pt](lr){3-8}
			\multicolumn{2}{c}{} & \multicolumn{2}{c}{$ 
				K=600$}&\multicolumn{2}{c}{$K=800$} 
			&\multicolumn{2}{c}{ 
				$K=1000$}\\  
			\cmidrule[0.5pt](lr){3-4} 
			\cmidrule[0.5pt](lr){5-6}\cmidrule[0.5pt](lr){7-8}
			\multirow{5}{*}{$1200$} &	Condat-V\~u  
			& 40.97 & 39717 & 38.48 & 37319 & 36.65 & 33408 
			\\ 
			& FPIHF & 14.04 & 3757 & 15.09 & 2942 & 29.44 & 
			4420
			\\ 
			& FHRB & 
			$\infty$ & $\infty$ & $\infty$ & $\infty$& 
			$\infty$ & 
			$\infty$ \\	 
			& Algorithm~\ref{algo:MTPD} & 5.61 & 2232 & 8.75 
			& 
			2548 & 15.27 & 3472\\					& 
			Algorithm~\ref{algo:CMTPD} & 9.53 & 3904 & 10.81 
			& 
			3190 & 13.35 & 3003
		\end{tabular}
	\end{table}

	\begin{table}
		\caption{Comparison of Condat-V\~u, FPIHF, FHRB,  
			Algorithm~\ref{algo:MTPD}, and  
			Algorithm~\ref{algo:CMTPD} for 
			$\kappa=1/20$.\label{table:II}}
		\begin{tabular}{cccccccccc} 
			& 	
			
			&  Av. time (s)& Av. iter & Av. 
			time 
			(s)& 
			Av. iter & 
			Av. time (s)& Av. iter\\ 
			\cmidrule[0.5pt](lr){3-8}
			$N$& Algorithm & \multicolumn{2}{c}{$ 
				K=200$}&\multicolumn{2}{c}{$K=250$} 
			&\multicolumn{2}{c}{ 
				$K=300$}\\ 	\cmidrule[0.5pt](lr){1-1} 
			\cmidrule[0.5pt](lr){2-2}
			\cmidrule[0.5pt](lr){3-4} 
			\cmidrule[0.5pt](lr){5-6}\cmidrule[0.5pt](lr){7-8}
			\multirow{5}{*}{$400$} &Condat-V\~u  
			& 1.17  & 8646 & 0.80 & 7055 & 0.71 & 6963 \\ 
			& FPIHF  & 1.88 & 3812  & 1.00 & 2310 & 0.81 & 
			1450 
			\\ 
			& FHRB & 
			5.07 
			&  41571 & 3.88 & 33493 & 3.33 & 
			29695 	\\
			& Algorithm~\ref{algo:MTPD} & 0.67 & 2176 & 0.48 
			& 
			1373& 0.40 & 874 \\ 
			& Algorithm~\ref{algo:CMTPD} & 1.38 & 4645 & 0.84 
			& 
			3080 & 0.65 & 1975 \\\cmidrule[0.5pt](lr){3-8}
			\multicolumn{2}{c}{} & \multicolumn{2}{c}{$ 
				K=400$}&\multicolumn{2}{c}{$K=500$} 
			&\multicolumn{2}{c}{ 
				$K=600$}\\  
			\cmidrule[0.5pt](lr){3-4} 
			\cmidrule[0.5pt](lr){5-6}\cmidrule[0.5pt](lr){7-8}
			\multirow{5}{*}{$800$} &Condat-V\~u  
			& 6.75  & 14402 & 6.26 & 12745  & 6.32 & 12662 
			\\ 
			& FPIHF & 7.30 & 5258 & 10.23 & 4952 & 12.18 & 
			4596
			\\ 
			& FHRB & 
			$\infty$ &  $\infty$ &  $\infty$ & $\infty$ &  
			21.98 & 
			49057\\
			& Algorithm~\ref{algo:MTPD} & 2.36 & 2488 & 4.64 
			& 
			3241 & 3.41 & 2156\\
			& Algorithm~\ref{algo:CMTPD} & 4.22 & 4567 & 
			6.93
			& 
			4827 & 7.86 & 4547 \\			 
			\cmidrule[0.5pt](lr){3-8}
			\multicolumn{2}{c}{} & \multicolumn{2}{c}{$ 
				K=600$}&\multicolumn{2}{c}{$K=800$} 
			&\multicolumn{2}{c}{ 
				$K=1000$}\\  
			\cmidrule[0.5pt](lr){3-4} 
			\cmidrule[0.5pt](lr){5-6}\cmidrule[0.5pt](lr){7-8}
			\multirow{5}{*}{$1200$} &Condat-V\~u  
			& 23.07 & 20395  & 21.38 & 19557 &  20.73 & 18875 
			\\ 
			& FPIHF & 16.15 & 3883  &24.68 & 4560 & 27.88 & 
			4347 
			\\ 
			& FHRB & 
			$\infty$ & $\infty$  & $\infty$ & $\infty$ & 
			$\infty$ 
			& $\infty$	\\
			& Algorithm~\ref{algo:MTPD} & 8.89 & 3240 & 17.18 
			& 
			4698 & 12.55 & 2997	\\
			& Algorithm~\ref{algo:CMTPD} & 13.32 & 4684 & 
			18.65 
			& 5355 & 18.66 & 4415 		
		\end{tabular}
	\end{table}
	
	\begin{table}
		\caption{Comparison of Condat-V\~u, FPIHF, FHRB,  
			Algorithm~\ref{algo:MTPD}, and  
			Algorithm~\ref{algo:CMTPD} for 
			$\kappa=1/30$.\label{table:III}}
		\begin{tabular}{cccccccccc} 
			& 			
			&  Av. time (s)& Av. iter & Av. 
			time 
			(s)& 
			Av. iter & 
			Av. time (s)& Av. iter\\ 
			\cmidrule[0.5pt](lr){3-8}
			$N$& Algorithm & \multicolumn{2}{c}{$ 
				K=200$}&\multicolumn{2}{c}{$K=250$} 
			&\multicolumn{2}{c}{ 
				$K=300$}\\ 	\cmidrule[0.5pt](lr){1-1} 
			\cmidrule[0.5pt](lr){2-2}
			\cmidrule[0.5pt](lr){3-4} 
			\cmidrule[0.5pt](lr){5-6}\cmidrule[0.5pt](lr){7-8}
			\multirow{5}{*}{$400$} &	Condat-V\~u  
			& 1.02 & 8039 & 0.83 & 7662 & 0.99 & 10701  \\ 
			& FPIHF  & 1.79 & 3479  & 1.40 & 2747 & 2.05 & 
			4157\\ 
			& FHRB & 
			2.80 & 27355 & 3.31 & 28091 & 3.30 & 27674 \\ 
			& Algorithm~\ref{algo:MTPD} & 0.50 & 1716 & 0.74 
			& 
			1904 & 0.83 & 2313\\
			& Algorithm~\ref{algo:CMTPD} & 1.19 & 3810 & 1.17 
			& 
			4102 & 1.49 & 4931  				
			\\ \cmidrule[0.5pt](lr){3-8}
			\multicolumn{2}{c}{} & \multicolumn{2}{c}{$ 
				K=400$}&\multicolumn{2}{c}{$K=500$} 
			&\multicolumn{2}{c}{ 
				$K=600$}\\  
			\cmidrule[0.5pt](lr){3-4} 
			\cmidrule[0.5pt](lr){5-6}\cmidrule[0.5pt](lr){7-8}
			\multirow{5}{*}{$800$} &	Condat-V\~u  
			& 5.94 & 11858 & 5.55 & 11729  & 4.95  & 11438\\ 
			& FPIHF & 7.68 & 5244  &8.58  & 3992 & 9.41 & 
			4396\\ 
			& FHRB &  
			$\infty$ &  $\infty$ & $\infty$  &  $\infty$ & 
			21.93 
			& 49864 \\ 
			& Algorithm~\ref{algo:MTPD} & 4.09 & 4037 & 4.96 
			& 
			3562 & 4.80 & 3072 \\
			& Algorithm~\ref{algo:CMTPD} & 5.14 & 5544 & 7.30 
			& 
			5238 & 8.03 & 5213  				
			\\ 			
			\cmidrule[0.5pt](lr){3-8}
			\multicolumn{2}{c}{} & \multicolumn{2}{c}{$ 
				K=600$}&\multicolumn{2}{c}{$K=800$} 
			&\multicolumn{2}{c}{ 
				$K=1000$}\\  
			\cmidrule[0.5pt](lr){3-4} 
			\cmidrule[0.5pt](lr){5-6}\cmidrule[0.5pt](lr){7-8}
			\multirow{5}{*}{$1200$} &	Condat-V\~u  
			& 17.98 & 16519 & 17.66  &  15903 & 17.64 & 15796 
			\\ 
			& FPIHF & 15.74 & 4064 & 40.18 & 7880 & 25.21 & 
			3818\\ 
			& FHRB & 
			$\infty$ & $\infty$  &  $\infty$ & $\infty$  & 
			$\infty$ & $\infty$  \\ 
			& Algorithm~\ref{algo:MTPD} & 9.29 & 3628 & 23.38 
			& 
			6338  & 9.82 & 2060 \\ 
			& Algorithm~\ref{algo:CMTPD} & 12.79 & 5127 & 
			29.25 & 
			8470 & 15.49 & 3463 				
			\\ 
		\end{tabular}
	\end{table}
	
	\begin{table}
		\caption{Average number of iterations per
			average time of iterations for FPIHF, 
			Algorithm~\ref{algo:MTPD}, and 
			Algorithm~\ref{algo:CMTPD}. We consider all 
			cases, 
			independently of $\kappa$.\label{table:IV}}
		\begin{tabular}{ccccc} 
			& 			
			&  Av. iter/Av. time (s)  & Av. iter/Av. time (s) 
			& Av. iter/Av. time (s) \\ 
			\cmidrule[0.5pt](lr){3-5}
			$N$& Algorithm & $ 
			K=200$&$K=250$ 
			&$K=300$\\ 	\cmidrule[0.5pt](lr){1-1} 
			\cmidrule[0.5pt](lr){2-2}
			\cmidrule[0.5pt](lr){3-3} 
			\cmidrule[0.5pt](lr){4-4}\cmidrule[0.5pt](lr){5-5}
			\multirow{2}{*}{$400$} & FPIHF  & 2000 & 2039  & 
			1550 
			\\ 
			& Algorithm~\ref{algo:MTPD} & 3289 & 2747 & 
			2361	\\
			& Algorithm~\ref{algo:CMTPD} & 3331 & 3409 & 2482
			\\ \cmidrule[0.5pt](lr){3-5}
			\cmidrule[0.5pt](lr){3-5}
			&  & $ 
			K=400$&$K=500$ 
			&$K=600$\\ 	
			\cmidrule[0.5pt](lr){3-3} 
			\cmidrule[0.5pt](lr){4-4}\cmidrule[0.5pt](lr){5-5}
			\multirow{2}{*}{$800$} & FPIHF  & 705 & 480 & 400
			\\ 
			& Algorithm~\ref{algo:MTPD} & 1047 & 712 & 592\\
			& Algorithm~\ref{algo:CMTPD} & 1125 & 706 & 601	
			\\ \cmidrule[0.5pt](lr){3-5}				
			\cmidrule[0.5pt](lr){3-5}
			& & $ 
			K=600$&$K=800$ 
			&$K=1000$\\ 
			\cmidrule[0.5pt](lr){3-3} 
			\cmidrule[0.5pt](lr){4-4}\cmidrule[0.5pt](lr){5-5}
			\multirow{2}{*}{$1200$} & FPIHF  & 254 & 192  & 
			152
			\\ 
			& Algorithm~\ref{algo:MTPD} & 382 & 275 & 226\\
			& Algorithm~\ref{algo:CMTPD} & 384 & 289 & 229
		\end{tabular}
\end{table}}

\subsection{Computed Tomography}
In this section we compare the numerical performance of Condat-V\~u, FPIHF, FHRB, Algorithm~\ref{algo:MTPD}, and 
	Algorithm~\ref{algo:CMTPD} in the context of Computed Tomography \cite{Kak2001}, i.e., in Problem~\ref{prob:numericproblem}, we aim to recover and image $\overline{x}$ from its noisy tomographic projection $z$. In this context, the linear operator $M$ represent the discretized Radon projector and the vector $z$ is a tomographic projection including a multiplicative Poisson noise of mean $M\overline{x}$. As a test image $\overline{x}$, we consider the image of $64\times 64$ pixels in Figure~\ref{fig:original} generated by the function {\it phantom} from MATLAB. The projector $M$ is given by the line length ray-driven projector 
	\cite{zengRayDriven93} and implemented in MATLAB using the line 
	fan-beam projector provided by the ASTRA toolbox 
	\cite{vanAarleASTRA2016,vanAarleASTRA2015}. We test different scenarios for the projector $M \in \R^{K\times N}$ varying the numbers of projections and the size of the detector, thus, $N=4096$ and $K \in \{1600,2790, 3375, 4500, 3375, 4500, 6090, 7440\}$. In all cases, the projector $L$ describes a fan-beam geometry over 180$^o$, the source-to-object distance is $800$ mm, and the source-to-image distance is $1200$ mm. In this case, the discrete gradient $\nabla\colon x\mapsto \nabla x=(\nabla_1x,\nabla_2x)$ includes horizontal and vertical differences through linear operators $\nabla_1$ and $\nabla_2$, respectively, its adjoint $\nabla^*$ is the discrete divergence, and we have $\|L\|=\sqrt{8}$ \cite{TV-chambolle}. To test the algorithms we consider $\alpha_1=1 $ and $\alpha_2=0.01 $. The step sizes were taken as in previous section. As a stop criterion we consider a tolerance of $10^{-8}$ for the relative error and a maximum of $10^4$ iterations.
	The result are shown in Table~\ref{table:CT}. From this table we observe that Algorithm~\ref{algo:MTPD} and 
	Algorithm~\ref{algo:CMTPD} overcome Condat-V\~u, FPIHF, and FHRB in terms of CPU Time, Number of Iterations, Objective Value, and PSNR, for every $K \in \{2790, 3375, 4500, 3375, 4500, 6090, 7440\}$. Additionally, we observe that as long as $K$ increases, the predominance of Algorithm~\ref{algo:MTPD} and Algorithm~\ref{algo:CMTPD} over the other methods also increases. In the particular case when $K=1600$, FPIHF exhibits the best performance, however, Algorithm~\ref{algo:MTPD} and 
	Algorithm~\ref{algo:CMTPD} provides smaller Objective Value and higher PSNR. Note that Condat-V\~u and FHRB, in all scenarios, is stopped by the maximum number of iteration. The Noisy sinogram and the recovered images for $K=7400$ are shown in Figure~\ref{fig:sino} and Figure~\ref{fig:reco}, respectively.
	{\small
		\begin{table}
			\caption{Comparison of Condat-V\~u, FPIHF, FHRB,  
					Algorithm~\ref{algo:MTPD}, and  
					Algorithm~\ref{algo:CMTPD} in 
					computed tomography problems.\label{table:CT}}
			\begin{tabular}{ccccccccccc} 
				$N=4096$  &  Time& Iter & O.V. & error & PSNR & Time& Iter & O.V. & error & PSNR \\ 
				\cmidrule[0.5pt](lr){2-11}
				Algorithm & \multicolumn{5}{c}{$ 
					K=1600$}&\multicolumn{5}{c}{$K=2790$} \\
				\cmidrule[0.5pt](lr){1-1}
				\cmidrule[0.5pt](lr){2-6} 
				\cmidrule[0.5pt](lr){7-11}
				Condat-V\~u  
				& 127.9  & $10^{4}$ & 4.91 & $5e{-5}$ &  6.0 & 195.7 & $10^{4}$ & 5.89 & $4e{-5}$& 8.5 \\ 
				FPIHF  & 61.9 & 1407 & 3.93 & $9e{-9}$ & 9.7 & 
				36.2 & 874 & 4.38 & $9e{-9}$ & 20.8
				\\ 
				FHRB & 
				130.1 
				&  $10^{4}$ & 11.0 & $5e{-6}$ & 3.1 & 
				181.8 & $10^{4}$ & 8.82 & $1e{-5}$& 6.2	\\
				Algorithm~\ref{algo:MTPD} & 113.8 & 4296 & 3.83 
				& 
				$9e{-9}$& 11.38 & 31.3 & 1092 & 4.31 & $9e{-9}$ & 24.3 \\ 
				Algorithm~\ref{algo:CMTPD} & 122.9 & 4446 & 3.83 & $9e{-9}$ 
				& 
				11.38 & 30.3 & 1084 & 4.31 & $9e{-9}$& 24.3  \\\cmidrule[0.5pt](lr){2-11}
				& \multicolumn{5}{c}{$K=3375$}&\multicolumn{5}{c}{$K=4500$} \\
				\cmidrule[0.5pt](lr){2-6} 
				\cmidrule[0.5pt](lr){7-11}
				Condat-V\~u  
				& 222.7  & $10^4$ & 5.53 & $4e{-5}$  & 7.8 & 216.8 & $10^4$ & 6.08 &$2e{-5}$  & 8.3
				\\ 
				FPIHF & 33.3 & 845 & 4.59 & $9e{-9}$ & 18.2 & 
				28.2 & 712 & 5.04 & $9e{-9}$ & 22.4
				\\ 
				FHRB & 
				221.8 &  $10^4$ &  7.26 & $1e{-5}$ &  
				6.0 & 
				225.4 & $10^4$ & 7.39 & $1e{-5}$ & 6.2 \\
				Algorithm~\ref{algo:MTPD} & 26.1 & 944 & 4.53 
				& 
				$9e{-9}$ & 20.5 & 21.2 & 779 & 4.98 
				& 
				$9e{-9}$ & 24.2 \\
				Algorithm~\ref{algo:CMTPD} & 27.6 & 980 & 4.53 
				& 
				$9e{-9}$ & 20.5 & 21.5 & 785 & 4.98 & $9e{-9}$ & 24.2	\\\cmidrule[0.5pt](lr){2-11}
				& \multicolumn{5}{c}{$K=6090$}&\multicolumn{5}{c}{$K=7440$} \\
				\cmidrule[0.5pt](lr){2-6} 
				\cmidrule[0.5pt](lr){7-11}
				Condat-V\~u  
				& 237.5  & $10^4$ & 6.94 & $2e{-5}$ & 10.9 & 250.2 & $10^4$ & 7.12 & $1e{-5}$ & 13.1
				\\ 
				FPIHF & 40.6 & 964 & 5.82 & $9e{-9}$ & 27.5 & 
				45.1 & 1014 & 6.42 & $9e{-9}$ & 30.6
				\\ 
				FHRB & 
				243.1 &  $10^4$ &  7.62 & $1e{-5}$ &  
				9.7 & 
				292.2 & $10^4$ & $7.29$ & $1e{-5}$ & 12.1 \\
				Algorithm~\ref{algo:MTPD} & 18.0 & 658 & 5.78
				& 
				$9e{-9}$ & 29.0 & 20.6 & 665 & 6.41 & $9e{-9}$ & 32.0 \\
				Algorithm~\ref{algo:CMTPD} & 16.5 & 661 & 5.79
				& 
				$9e{-9}$ & 29.0 & 20.6 & 676 & 6.41 & $9e{-9}$ & 32.0	
			\end{tabular}
		\end{table}
	}

	\begin{figure}
		\centering
		\subfloat[ Original Image ($64\times 64 $ pixels)]{\label{fig:original}\includegraphics[scale=0.28]{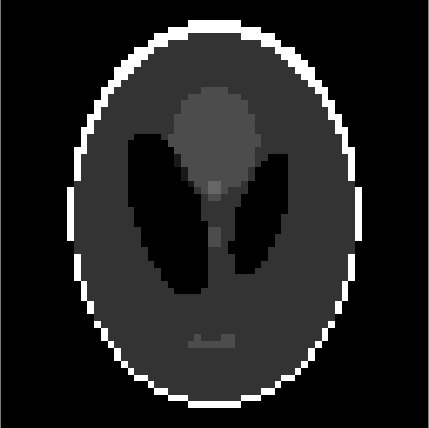}}\,
		\subfloat[Noisy Sinogram ($K=7400$)
		]{\label{fig:sino}\includegraphics[scale=0.25]{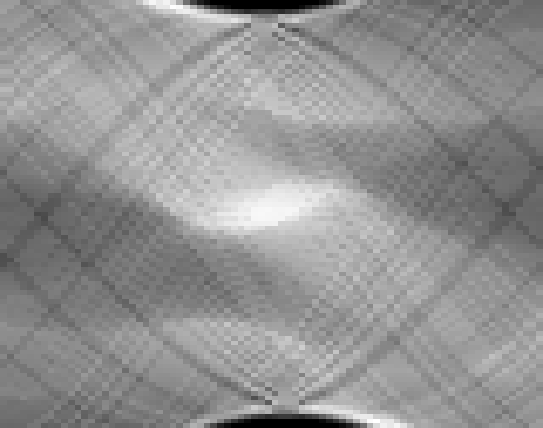}}
		\caption{ }\label{fig:original_sino}
	\end{figure}
	
	\begin{figure}
		\centering
		\subfloat[ Condat-Vu (PSNR= 13.1)]{\label{fig:CV}\includegraphics[scale=0.28]{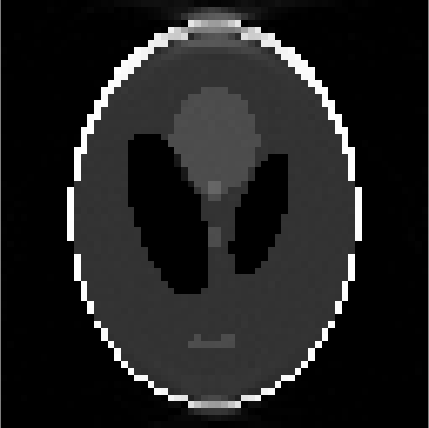}}\,
		\subfloat[FPIHF (PSNR = 30.6)
		]{\label{fig:BAD}\includegraphics[scale=0.28]{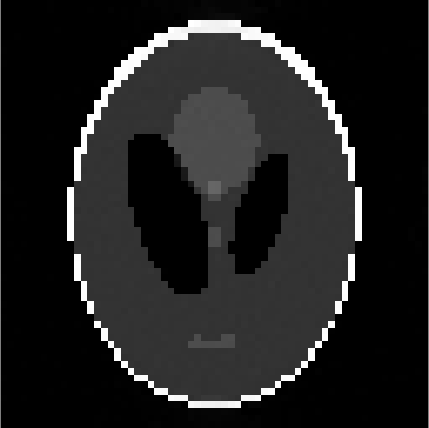}}\,
		\subfloat[FHRB (PSNR = 12.1) 
		]{\label{fig:MT}\includegraphics[scale=0.28]{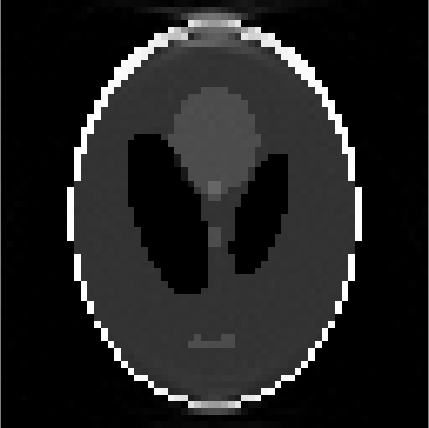}}\\
		\subfloat[Algorithm~\ref{algo:MTPD} (PSNR=32.0)
		]{\label{fig:MTPI}\includegraphics[scale=0.28]{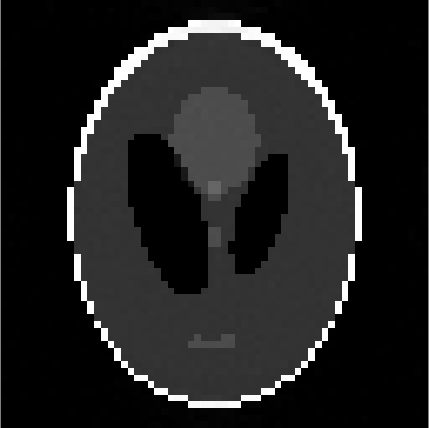}}\,
		\subfloat[Algorithm~\ref{algo:CMTPD} (PSNR=32.0)]{\label{fig:SDRPI}\includegraphics[scale=0.28]{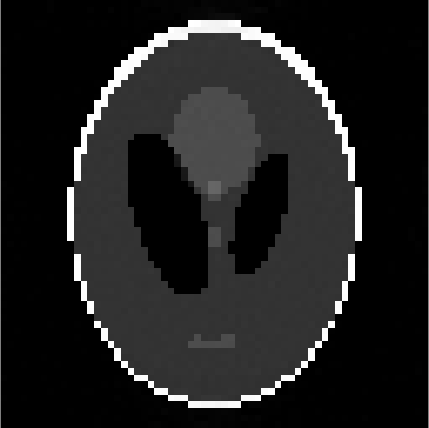}}
		\caption{Recovered images for $K=7400$.}\label{fig:reco}
	\end{figure}

\section{Conclusions}
In this article, we proposed two methods for solving monotone 
inclusion involving a maximally monotone operator, a 
cocoercive operator, a monotone-Lipschitz operator, and a normal 
cone 
to a vector subspace. Our methods offer several advantages in 
computational implementations compared to the method proposed in 
\cite{BricenoRoldanYuchaoChen} for solving the same problem. 
Specifically, our methods require only one 
activation of the Lipschitzian operator and two projections onto 
the 
vector subspace per iteration. We further applied our methods to 
systems of linear composite primal-dual problems, inclusions involving a sum of maximally monotone operator, and optimization problems, 
leveraging the properties of our main methods. The experimental 
results demonstrate the numerical superiority of our 
method over existing methods in the literature. Overall, our 
proposed method provides an efficient and effective solution for 
monotone inclusion problems, offering computational advantages 
and exhibiting promising performance in applications.

\end{document}